\documentclass{amsart}

\usepackage{stmaryrd}
\usepackage{multicol}
\usepackage{enumerate}
\usepackage[pdftex]{graphicx}
\usepackage{amsmath}
\usepackage{amsthm}

\newtheorem{thm}{Theorem}[section]
\newtheorem{prop}[thm]{Proposition}
\newtheorem{lem}[thm]{Lemma}
\newtheorem{cor}[thm]{Corollary}

\newtheorem{conj}[thm]{Conjecture} 

\theoremstyle{definition}
\newtheorem{definition}[thm]{Definition}

\theoremstyle{remark}
\newtheorem{notation}[thm]{Notation}

\numberwithin{equation}{section}

\newcommand{\Span}{Span} 
\newcommand{\Hom}{Hom} 

\newcommand{\id}{\mathrm{id}}

\newcommand{\Pmc}{\mathcal{P}} 
\newcommand{\Cmc}{\mathcal{C}} 
 
\newcommand{\Bmc}{\mathcal{B}} 
\newcommand{\Fmc}{\mathcal{F}}
\newcommand{\Tmc}{\mathcal{T}} 
\newcommand{\Xmc}{\mathcal{X}} 
\newcommand{\Dbbb}{\mathbb{D}} 
\newcommand{\Hbbb}{\mathbb{H}} 
\newcommand{\Pbbb}{\mathbb{P}} 
\newcommand{\Rbbb}{\mathbb{R}} 
 
\newcommand{\Cbbb}{\mathbb{C}} 
\newcommand{\Mtd}{\widetilde{M}} 
\newcommand{\etd}{\widetilde{e}}

\begin{document}

\title{Collar lemma for Hitchin Representations.}

\begin{abstract}
In this article, we prove an analog of the classical collar lemma in the
setting of Hitchin representations.
\end{abstract}

\author{Gye-Seon Lee$^\dagger$, Tengren Zhang$^\ddagger$}

\thanks{
\hspace{-0.2in}$\dagger$: Gye-Seon Lee was supported by the DFG research grant ``Higher Teichm\"uller Theory".\\
$\ddagger$: Tengren Zhang was partially supported by U.S. National Science Foundation grants DMS 1006298, DMS 1306992 and DMS 1307164. \\
The authors acknowledge support from U.S. National Science Foundation grants DMS 1107452, 1107263, 1107367 ``RNMS: GEometric structures And Representation varieties" (the GEAR Network).}

\maketitle

\section{Introduction.} 

Let $S$ be a closed oriented topological surface and let $\Gamma$ be its fundamental group. The Teichm\"uller space of $S$, denoted $\Tmc(S)$, is the space of hyperbolic structures on $S$. Via the holonomy representation, $\Tmc(S)$ can be identified with a component of the space of conjugacy classes of representations from $\Gamma$ to $PSL(2,\Rbbb)$. One advantage of doing so is that it allows us to generalize $\Tmc(S)$ in the following way. It is a standard fact in representation theory that for any $n\geq 2$, there is a unique (up to conjugation) irreducible representation $\iota_n:PSL(2,\Rbbb)\to PSL(n,\Rbbb)$. This gives, via post composition, an embedding
\[\Tmc(S)\hookrightarrow\Xmc_n(S):=\Hom(\Gamma,PSL(n,\Rbbb))/PSL(n,\Rbbb).\]
The image of this embedding is known as the \emph{Fuchsian locus} and the component of $\Xmc_n(S)$ containing the Fuchsian locus is the \emph{$n$th-Hitchin component}, denoted $Hit_n(S)$. By definition, $Hit_2(S)=\Tmc(S)$, so Hitchin representations can be thought of as generalizations of Fuchsian representations.

For the hyperbolic structures in $\Tmc(S)$, there is a classical result first due to Keen \cite{Kee1} known as the collar lemma. It gives an effective lower bound on the width of a collar neighborhood of a simple closed curve in a hyperbolic surface, which grows to $\infty$ as the length of the simple closed curve is shrunk to $0$. A consequence of the collar lemma is that if two closed curves $\gamma$ and $\eta$ in a hyperbolic surface have non-vanishing geometric intersection number and $\gamma$ is simple, then there is an explicit lower bound on the length of $\eta$ in terms of the length of $\gamma$. This is a powerful tool that has been used to understand surfaces. For example, it was used to study the length spectrum of Riemann surfaces (see Buser \cite{Bus1}).

The goal of this paper is to generalize a version of the classical collar lemma to Hitchin representations. In this setting, the width of a collar neighborhood is not well defined since Hitchin representations in general do not give a metric on $S$. However, for every Hitchin representation $\rho$, we do still have a natural notion of length for homotopy classes of closed curves in $S$. 

By Labourie \cite{Lab1}, we know that for any Hitchin representation $\rho$ and any non-identity element $X$ in $\Gamma$, $\rho(X)$ is diagonalizable over $\Rbbb$ with eigenvalues that have pairwise distinct moduli. Hence, given any representation $\rho$ in $Hit_n(S)$ and any closed curve $\gamma$ in $S$, we can define the \emph{$\rho$-length} of $\gamma$ to be
\[l_\rho(\gamma)=\log\bigg|\frac{\lambda_n}{\lambda_1}\bigg|,\]
where $\lambda_n$ and $\lambda_1$ are the eigenvalues of $\rho(X)$ of largest and smallest modulus respectively, and $X\in\Gamma$ corresponds to the curve $\gamma$ equipped with a choice of orientation. Observe that the $\rho$-length does not depend on the choice of orientation on $\gamma$, and is constant on each homotopy class of closed curves in $S$. 

In the case when $\rho\in Hit_2(S)$, $l_\rho(\gamma)$ is exactly the hyperbolic length of the geodesic homotopic to $\gamma$, measured in the hyperbolic metric corresponding to $\rho$. Also, Choi-Goldman \cite{ChoGol1} proved that representations in $Hit_3(S)$ are exactly holonomies of convex $\Rbbb\Pbbb^2$ structures on $S$. Moreover, each such convex $\Rbbb\Pbbb^2$ structure also induces a natural Finsler metric, known as the Hilbert metric, on $S$. One can then verify that in the case when $\rho\in Hit_3(S)$, $l_\rho(\gamma)$ is the length of the geodesic homotopic to $\gamma$, measured in the Hilbert metric induced by the convex $\Rbbb\Pbbb^2$ structure corresponding to $\rho$.

With this, we have the following theorem, which one can think of as a generalization of the collar lemma.

\begin{thm}\label{main theorem}
Let $S$ be a surface of genus $g\geq 2$, and let $\gamma$, $\eta$ be two non-contractible closed curves in $S$. Denote the geometric intersection number between $\gamma$ and $\eta$ by $i(\eta,\gamma)$. Then, for any $n\geq 2$ and any $\rho\in Hit_n(S)$, the following hold:
\begin{enumerate}
\item If $i(\eta,\gamma)\neq 0$, then
\[\frac{1}{\exp(l_\rho(\eta))}< 1-\frac{1}{\exp(\frac{l_\rho(\gamma)}{n-1})}.\]
\item If $i(\eta,\gamma)\neq 0$ and $\gamma$ is simple, then  there are non-negative integers $u$, $v$ with $u\geq v$ and $u+v=i(\eta,\gamma)$ so that
\[\frac{1}{\exp(l_\rho(\eta))}< \bigg(1-\frac{1}{\exp(\frac{l_\rho(\gamma)}{n-1})}\bigg)^u\bigg(1-\frac{1}{\exp(l_\rho(\gamma))}\bigg)^v.\]
\item Let $\delta_n>0$ be the unique real solution to the equation $e^{-x}+e^{-x/(n-1)}=1$. If $\eta$ is a non-simple closed curve, then
\[l_\rho(\eta)>\delta_n.\]
\end{enumerate}
\end{thm}

We can say what the constants $u$ and $v$ are in (2) of Theorem \ref{main theorem}. Choose orientations on $\gamma$ and $\eta$, and let $\hat{i}(\gamma,\eta)$ be the algebraic intersection number between $\gamma$ and $\eta$. Then 
\[u=\frac{i(\gamma,\eta)+|\hat{i}(\gamma,\eta)|}{2}\text{ and }v=\frac{i(\gamma,\eta)-|\hat{i}(\gamma,\eta)|}{2}.\]
More geometrically, if we let $w_0$ and $w_1$ be number of times $\eta$ crosses $\gamma$ from left to right and from right to left respectively, then $u=\max\{w_0,w_1\}$ and $v=\min\{w_0,w_1\}$. Observe also that $\lim_{n\to\infty}\delta_n=\infty$ and $\delta_2=\log(2)$.

Theorem \ref{main theorem} is in fact a consequence of a more general inequality between $l_\rho(\gamma)$ and the eigenvalues of the image of the group element corresponding to $\eta$ under the representation $\rho$ (see Proposition \ref{main proposition}).

In the case of $\Tmc(S)$, the first inequality in Theorem \ref{main theorem} can be rewritten as 
\[(\exp(l_\rho(\eta))-1)(\exp(l_\rho(\gamma))-1)> 1.\] 
This is weaker than a version of the classical collar lemma, which is the inequality 
\[\sinh\bigg(\frac{l_\rho(\eta)}{2}\bigg)\sinh\bigg(\frac{l_\rho(\gamma)}{2}\bigg)> 1,\]
although in both inequalities, $l_\rho(\eta)$ grows logarithmically with $\frac{1}{l_\rho(\gamma)}$. Furthermore, while the classical collar lemma is sharp, we are unable to prove the same for Theorem \ref{main theorem}. This led us to conjecture, in an earlier version of this paper, that for any $\rho\in Hit_n(S)$, there is some representation $\rho'$ in the Fuchsian locus of $Hit_n(S)$ such that $l_\rho(\gamma)\geq l_{\rho'}(\gamma)$ for all $\gamma\in\Gamma$. If the conjecture holds, then we can obtain a sharp version of Theorem \ref{main theorem}. Recently, Tholozan \cite{Tho1} proved that this conjecture is true for $n=3$, but F. Labourie showed that it fails for all $n\geq 4$. See Section \ref{comparison} for more details.

Choi \cite{Cho1} proved an analog of the Margulis lemma for convex $\Rbbb\Pbbb^2$ surfaces. As a consequence, he showed the existence of a collar neighborhood in the convex $\Rbbb\Pbbb^2$ surface about a simple closed curve of sufficiently short length, and found (non-explicit) lower bounds for the width of this collar neighborhood in terms of the length of the simple closed curve. This analog of the Margulis lemma was later extended by Cooper-Long-Tillman \cite{CooLonTil1} to all convex real projective manifolds. Burger-Pozzetti \cite{BurPoz1} also recently proved a similar result for maximal representations into $Sp(2n,\Rbbb)$.

Although the inequalities in Theorem \ref{main theorem} depend on $n$, we can obtain as a corollary the following ``universal collar lemma" that is simultaneously true for all Hitchin representations. 

\begin{cor}[Corollary \ref{main corollary}]\label{main corollary intro}
Let $S$ be a surface of genus $g\geq 2$, and let $\gamma$, $\eta$ be two non-contractible closed curves in $S$. Then for any $n\geq 2$ and any $\rho\in Hit_n(S)$, the following hold:
\begin{enumerate}
\item If $i(\eta,\gamma)\neq 0$, then
\[(\exp(l_\rho(\gamma))-1)(\exp(l_\rho(\eta))-1)> 1.\]
\item If $i(\eta,\gamma)\neq 0$ and $\gamma$ is simple, then  
\[(\exp(l_\rho(\gamma))-1)\bigg(\exp\bigg(\frac{l_\rho(\eta)}{ i(\eta,\gamma)}\bigg)-1\bigg)> 1.\]
\item If $\eta$ is a non-simple closed curve, then
\[l_\rho(\eta)>\log(2).\]
\end{enumerate}
\end{cor}

Unfortunately, for $\rho\in Hit_n(S)$ when $n\geq 4$, it is not known whether there exists a metric on $S$ that induces $l_\rho$ as its length function. However, we can still interpret Theorem \ref{main theorem} and Corollary \ref{main corollary intro} geometrically by considering the $SL(n,\Rbbb)$ symmetric space $\widetilde{M}$. Normalize the Riemannian metric on $\Mtd$ so that for any $Z\in PSL(n,\Rbbb)$ with real eigenvalues, 
\[\inf\{d_{\Mtd}(o,Z\cdot o):o\in\Mtd\}=\sqrt{2\sum_{i=1}^n(\log|\lambda_i|)^2},\]
where $\lambda_1,\dots,\lambda_n$ are the eigenvalues of $Z$ and $d_{\Mtd}$ is the distance function on $\Mtd$ induced by the normalized Riemannian metric. Let $M:=\rho(\Gamma)\backslash\widetilde{M}$, and for any closed curve $\omega$ in $M$, let $l_M(\omega)$ be the length of $\omega$ measured in the Riemannian metric on $M$ induced by the normalized Reimannian metric on $\Mtd$. Then the following corollary of Theorem \ref{main theorem} and Corollary 1.2 holds.

\begin{cor}[Corollary \ref{symmetric space}]
Let $\gamma$, $\eta$ be two non-contractible closed curves in $S$ and let $X$, $Y$ be elements in $\Gamma$ corresponding to $\gamma$, $\eta$ respectively. For any $\rho\in Hit_n(S)$, let $\gamma'$, $\eta'$ be two closed curves in $M$ that correspond to $X,Y\in\Gamma$ respectively. Then the statements in Theorem \ref{main theorem} and Corollary \ref{main corollary intro} hold, with $l_\rho(\gamma)$ and $l_\rho(\eta)$ replaced with $l_M(\gamma')$ and $l_M(\eta')$ respectively.
\end{cor}

It is an important remark that this corollary (and hence Theorem \ref{main theorem}) is not simply a quantitative version of the Margulis lemma on $PSL(n,\Rbbb)$ because the closed curves $\gamma'$ and $\eta'$ do not need to intersect, even when $i(\gamma,\eta)\neq 0$. Theorem \ref{main theorem} is a property that is special to Hitchin representations. In fact, for any pair of simple closed curves in $S$, one can find a sequence of quasi-Fuchsian representations 
\[\rho_i:\Gamma\to PSO(3,1)^+\subset PSL(4,\Rbbb)\] 
so that the lengths of the geodesics in $\rho_i(\Gamma)\backslash\Mtd$ corresponding to both of these two simple closed curves converge to $0$ along this sequence. In particular, Theorem \ref{main theorem} does not hold on the space of quasi-Fuchsian representations. This is explained in greater detail in Section \ref{counter example}. 

As a final consequence of Theorem \ref{main theorem}, we have the following properness result.

\begin{cor}[Corollary \ref{proper}]
Let $\Cmc:=\{\gamma_1,\dots,\gamma_k\}$ be a collection of closed curves in $S$ that contains a pants decomposition, so that the complement of $\Cmc$ in $S$ is a union of discs. Then the map
\begin{eqnarray*}
Hit_n(S)&\to&\Rbbb^k\\
\rho&\mapsto&(l_\rho(\gamma_1),\dots,l_\rho(\gamma_k))
\end{eqnarray*}
is proper.
\end{cor}

In other words, in order for a sequence $\{\rho_i\}_{i=1}^\infty$ in $Hit_n(S)$ to escape, the $\rho_i$-length of some curve in $\Cmc$ must grow to $\infty$. Refer to Section \ref{corollaries} for more corollaries of Theorem \ref{main theorem}.

{\bf Acknowledgements:} This work started as an attempt to answer a question that Scott Wolpert asked the second author in August 2013. He asked if a collar lemma exists for convex $\Rbbb\Pbbb^2$ surfaces. At that time, he and both authors were attending the ``Pressure metric and Higgs bundles" masterclass and conference held in QGM of Aarhus University. We are grateful to him for asking the question, and to QGM for their kind hospitality. We also thank Daniele Alessandrini, Ara Basmajian, Richard Canary, Francois Labourie and Anna Wienhard for many useful conversations with the authors. This work has benefitted from the GEAR network retreat 2014 and the GEAR junior retreat 2014, which were attended by both authors. It was also the focus of the GEAR graduate student internship during which the second author visited the first author at Heidelberg University. Finally, we would like to thank the referees for carefully reading this paper and suggesting several improvements.

\tableofcontents

\section{Proof of Theorem.}\label{Proof of Theorem}

In this section, we give the proof of Theorem \ref{main theorem}. We start by discussing some useful topological properties of $\Gamma$ and its boundary in Section \ref{boundary of the group}. Then for the sake of demonstrating the proof without too many technical details, we prove (1) of Theorem \ref{main theorem} for the special case of $Hit_3(S)$ in Section \ref{n=3 case}. Next, we develop the technical tools that we need in Section \ref{Hitchin properties}, and apply them in Section \ref{general} to prove Theorem \ref{main theorem} in its full generality. 

\subsection{Properties of the boundary of the group.}\label{boundary of the group}
It is well-known that $\Gamma:=\pi_1(S)$ is Gromov hyperbolic, so the Cayley graph of $\Gamma$ has a natural boundary, denoted by $\partial_\infty\Gamma$, and the action of $\Gamma$ on its Cayley graph extends to an action on $\partial_\infty\Gamma$. Moreover, if we choose $\rho\in\Tmc(S)$, i.e. a hyperbolic structure on $S$, we get a $\rho$-equivariant identification of $\partial_\infty\Gamma$ with the boundary of the Poincar\'e disc $\partial\Dbbb$. 

For any hyperbolic element $A\in PSL(2,\Rbbb)$, the \emph{axis} of $A$, denoted $L_A$, is the unique geodesic in $\Dbbb$ whose endpoints are the repelling and attracting fixed points of $A$ in $\partial\Dbbb$. The proof of the main theorem relies crucially on an important property of the action of $\Gamma$ on $\partial_\infty\Gamma$, which we state as Lemma \ref{topological}. These are well-known facts about surface groups, but for want of a good reference, we will give the proof here.

\begin{lem}\label{Ara}
Let $B$ and $B'$ be non-commuting elements in $PSL(2,\Rbbb)$ that generate a subgroup consisting only of hyperbolic isometries. If the translation lengths of $B$ and $B'$ are the same and $L_{B'}\cap L_B=\emptyset$, then $(B\cdot L_{B'})\cap L_{B'}=\emptyset$. 
\end{lem}

\begin{proof}
Since $B$ and $B'$ do not commute, $L_B\neq L_{B'}$. Since the commutator $[B,B']$ is not parabolic, $B$ and $B'$ cannot share a fixed point. Hence, by changing coordinates and replacing $B$ and $B'$ with their inverses if necessary, we can assume that  $L_B$ and $L_{B'}$ are as in Figure \ref{hyperbolic} and $B$, $B'$ translate along their axes in the directions drawn. 

Let $L$ be the geodesic in $\Hbbb^2$ that is perpendicular to both $L_{B'}$ and $L_B$, and let $R$ be the reflection about $L$. There is a unique geodesic $K$ that is perpendicular to $L_B$ and whose distances to $L$ and $B\cdot L$ are equal. Let $S$ be the reflection about $K$, and note that $B=SR$. Also, observe that the distance between $K$ and $L$ is realized only by the points $K\cap L_B$ and $L\cap L_B$, and is half the translation length of $B$, which we denote by $T$. Furthermore, $(B\cdot L_{B'})\cap L_{B'}=(SR\cdot L_{B'})\cap L_{B'}=(S\cdot L_{B'})\cap L_{B'}$ is empty if and only if $K\cap L_{B'}$ is empty.

Thus, it is sufficient to show that $K\cap L_{B'}$ is empty. Suppose for contradiction that it is not. As before, there is a unique geodesic $K'$ so that $B'=S'R$, where $S'$ is the reflection about $K'$. Since the translation lengths of $B$ and $B'$ are the same, the symmetry between $B$ and $B'$ ensures that  $K'\cap L_{B}$ is also nonempty. 

Now, note that $K'\cap L_{B'}$ lies between $K\cap L_{B'}$ and $L\cap L_{B'}$ because 
\[d(K\cap L_{B'},L\cap L_{B'})>d(K\cap L_B,L\cap L_B)=\frac{T}{2}=d(K'\cap L_{B'},L\cap L_{B'}).\] Similarly, $K\cap L_{B}$ lies between $K'\cap L_{B}$ and $L\cap L_{B}$. This implies that $K$ and $K'$ have a common point of intersection, $p$ (see Figure \ref{hyperbolic}). Observe that $B'B^{-1}=S'RR^{-1}S^{-1}=S'S$ fixes $p$, but that is impossible because $B'B^{-1}$ is not elliptic. 
\end{proof}

\begin{figure}
\includegraphics[scale=0.6]{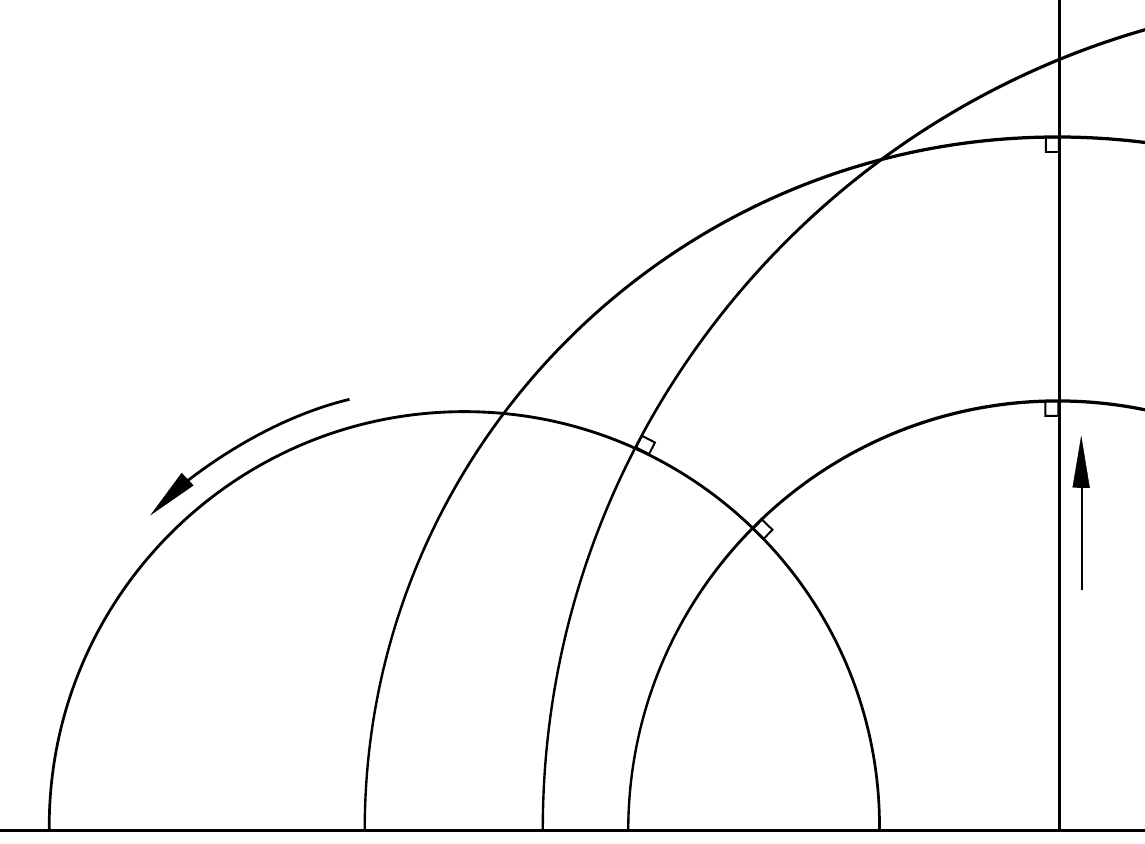}
\put (-247, 15){\makebox[0.7\textwidth][r]{$L_B$ }}
\put (-249, 55){\makebox[0.7\textwidth][r]{$B$ }}
\put (-439, 15){\makebox[0.7\textwidth][r]{$L_{B'}$ }}
\put (-400, 73){\makebox[0.7\textwidth][r]{$B'$ }}
\put (-295, 58){\makebox[0.7\textwidth][r]{$L$ }}
\put (-342, 97){\makebox[0.7\textwidth][r]{$K$ }}
\put (-312, 80){\makebox[0.7\textwidth][r]{$K'$ }}
\put (-288, 113){\makebox[0.7\textwidth][r]{$p$ }}
\caption{An impossible configuration of $K$ and $K'$ in Lemma \ref{Ara}}\label{hyperbolic}
\end{figure}
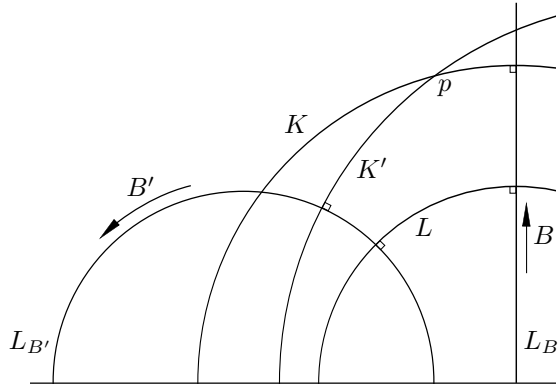

\begin{lem}\label{topological}
Let $A$, $B$, $B'$ be pairwise non-commuting elements in $\Gamma$ so that $B$ and $B'$ are conjugate. Let $a^+$, $b^+$, $b'^+$ be the attracting fixed points and $a^-$, $b^-$, $b'^-$ be the repelling fixed points of $A$, $B$, $B'$ respectively. If 
\[\ a^+,\ b'^+,\ b^+,\ a^-,\ b^-,\ b'^-\] 
lie in $\partial_\infty\Gamma$ in that cyclic order, then
\[a^+,\  b'^+,\  B\cdot a^+,\  b^+,\  a^-,\  b^-,\ B^{-1}\cdot a^+,\  b'^-\] 
lie in $\partial_\infty\Gamma$ in that cyclic order (see Figure \ref{Figure2}).
\end{lem}

\begin{proof}
Let $s_0$ be the open subsegment of $\partial_\infty\Gamma$ with endpoints $b'^-$ and $b^+$ that does not contain $b^-$, and let $s_1$ be the open subsegment of $\partial_\infty\Gamma$ with endpoints $b'^+$ and $b^+$ that does not contain $b^-$. Observe that $B\cdot b'^-$ lies in $s_0$ and $B\cdot b'^+$ lies in $s_1$. 

Choose a hyperbolic metric on $S$. This identifies $\partial_\infty\Gamma$ with $\partial\Dbbb$ and $\Gamma$ with a discrete, torsion-free subgroup of $PSL(2,\Rbbb)$. Since 
\[\ a^+,\ b'^+,\ b^+,\ a^-,\ b^-,\ b'^-\] 
lie in $\partial_\infty\Gamma$ in that cyclic order, $L_B$ and $L_{B'}$ have to be disjoint. Moreover, $B$ and $B'$ have the same translation lengths and do not commute. Hence, we can apply Lemma \ref{Ara} to conclude that $B\cdot L_{B'}$ and $L_{B'}$ are disjoint. This implies that both $B\cdot b'^-$ and $B\cdot b'^+$ have to lie in $s_1$. Since $a^+$ lies in $s_0$ between $b'^-$ and $b'^+$, $B\cdot a^+$ must lie in $s_1$ between $B\cdot b'^-$ and $B\cdot b'^+$. In particular, 
\[a^+,\  b'^+,\  B\cdot a^+,\  b^+,\  a^-\]
lie in $\partial_\infty\Gamma$ in that cyclic order (see Figure \ref{Figure2}). 

A similar argument, using $B^{-1}$ instead of $B$, shows that 
\[a^-,\  b^-,\ B^{-1}\cdot a^+,\  b'^-,\  a^+\]
lie in $\partial_\infty\Gamma$ in that cyclic order. This proves the lemma.
\end{proof}

\begin{figure}
\includegraphics[scale=0.9]{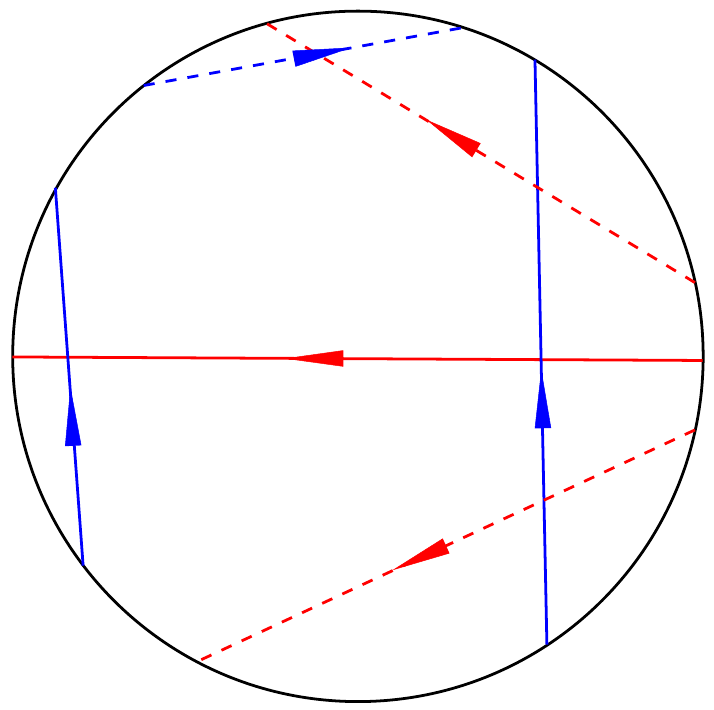}
\put (-380, 3){\makebox[0.7\textwidth][r]{$B^{-1}\cdot a^+$ }}
\put (-410, 30){\makebox[0.7\textwidth][r]{$b'^-$ }}
\put (-430, 88){\makebox[0.7\textwidth][r]{$a^+$ }}
\put (-400, 110){\makebox[0.7\textwidth][r]{$L_{B'}$ }}
\put (-417, 133){\makebox[0.7\textwidth][r]{$b'^+$ }}
\put (-369, 154){\makebox[0.7\textwidth][r]{$B\cdot b'^-$ }}
\put (-305, 166){\makebox[0.7\textwidth][r]{$B\cdot L_{B'}$ }}
\put (-356, 180){\makebox[0.7\textwidth][r]{$B\cdot a^+$ }}
\put (-289, 179){\makebox[0.7\textwidth][r]{$B\cdot b'^+$ }}
\put (-285, 171){\makebox[0.7\textwidth][r]{$b^+$ }}
\put (-224, 110){\makebox[0.7\textwidth][r]{$B\cdot a^-$ }}
\put (-237, 90){\makebox[0.7\textwidth][r]{$a^-$ }}
\put (-214, 70){\makebox[0.7\textwidth][r]{$B^{-1}\cdot a^-$ }}
\put (-280, 9){\makebox[0.7\textwidth][r]{$b^-$ }}
\put (-277, 35){\makebox[0.7\textwidth][r]{$L_B$ }}
\put (-270, 94){\makebox[0.7\textwidth][r]{$L_A$ }}
\put (-310, 138){\makebox[0.7\textwidth][r]{$B\cdot L_A$ }}
\put (-441, 135){\makebox[0.7\textwidth][r]{$s_0$ }}
\put (-377, 213){\makebox[0.7\textwidth][r]{$s_1$ }}
\caption{The cyclic order of the attracting and repelling fixed points of $A$, $B$, $B'$ and $BAB^{-1}$ along $\partial_\infty\Gamma$ in Lemma \ref{topological}.}\label{Figure2}
\end{figure}
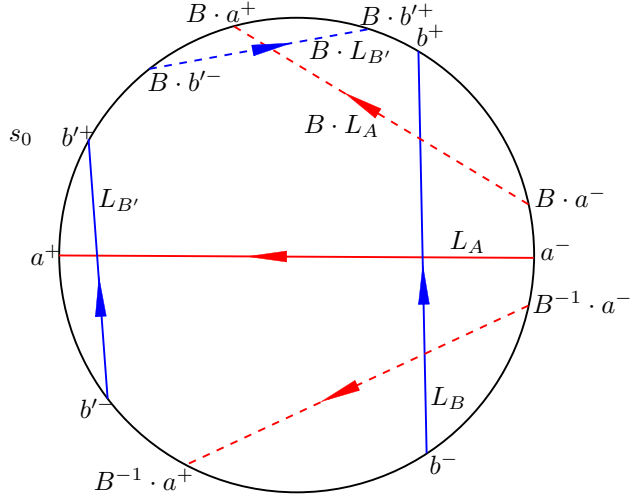

\subsection{Proof in the $PSL(3,\Rbbb)$ case.}\label{n=3 case}
In order to demonstrate the main ideas of the proof without involving too many technicalities, we will first prove (1) of Theorem \ref{main theorem} in the special case when $n=3$, i.e. $\rho:\Gamma\to PSL(3,\Rbbb)=SL(3,\Rbbb)$ is a Hitchin representation. 

By Choi-Goldman \cite{ChoGol1}, we know that in this case, $\rho$ is the holonomy of a convex $\Rbbb\Pbbb^2$ structure on $S$. In other words, there is a strictly convex domain $\Omega_\rho$ in $\Rbbb\Pbbb^2$ which is preserved by the $\Gamma$-action on $\Rbbb\Pbbb^2$ induced by $\rho$, and on which the $\Gamma$-action is properly discontinuous and cocompact. Moreover, $\rho(X)$ is diagonalizable with positive pairwise distinct eigenvalues for any non-identity element $X\in\Gamma$, (see Theorem 3.2 of Goldman \cite{Gol1}) so $\rho(X)$ has an attracting and repelling fixed point in $\partial\Omega_\rho$. Since the Hilbert metric in $\Omega_\rho$ is invariant under projective transformations and the geodesics of the Hilbert metric are lines, one can use the \v{S}varc-Milnor lemma (Proposition 8.19 of \cite{BriHae1}) to construct a continuous map 
\[\xi^{(1)}:\partial_\infty\Gamma\to\partial\Omega_\rho\]
which identifies the attracting fixed point of any $X\in\Gamma$ to the attracting fixed point of $\rho(X)$.

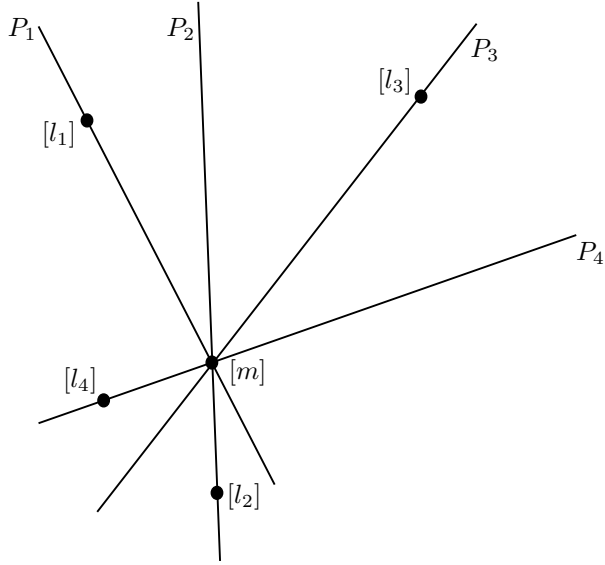
\begin{figure}
\includegraphics[scale=0.9]{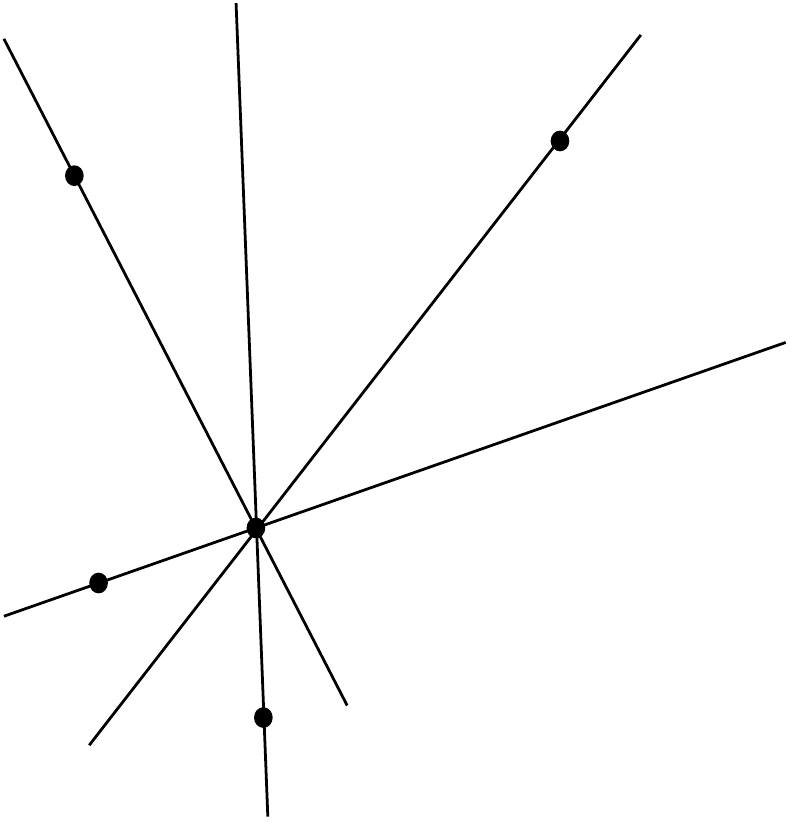}
\put (-432, 67){\makebox[0.7\textwidth][r]{$[l_4]$ }}
\put (-455, 200){\makebox[0.7\textwidth][r]{$P_1$ }}
\put (-440, 160){\makebox[0.7\textwidth][r]{$[l_1]$ }}
\put (-395, 200){\makebox[0.7\textwidth][r]{$P_2$ }}
\put (-370, 23){\makebox[0.7\textwidth][r]{$[l_2]$ }}
\put (-313, 180){\makebox[0.7\textwidth][r]{$[l_3]$ }}
\put (-280, 192){\makebox[0.7\textwidth][r]{$P_3$ }}
\put (-240, 115){\makebox[0.7\textwidth][r]{$P_4$ }}
\put (-368, 70){\makebox[0.7\textwidth][r]{$[m]$ }}
\caption{A choice of vectors $l_i$ to compute the cross ratio $(P_1,P_2,P_3,P_4)$.}\label{Figure4}
\end{figure}

Pick any four projective lines in $\Rbbb\Pbbb^2$ that intersect at a common point, such that no three of the four agree. There is a classical projective invariant of these four projective lines, called the \emph{cross ratio}, which can be defined as follows. Let the four projective lines be $P_1$, $P_2$, $P_3$, $P_4$ and let $m$ be a vector in $\Rbbb^3$ so that $[m]$, the projective point corresponding to the $\Rbbb$-span of $m$, is the common point of intersection of the $P_i$. For each $i$, choose a vector $l_i\in\Rbbb^3$ so that $[l_i]\neq[m]$ and $[l_i]$ lies in $P_i$ (see Figure \ref{Figure4}). By choosing a linear identification 
\[f:\bigwedge^3\Rbbb^3\to\Rbbb,\]
we can evaluate the expression
\[(P_1,P_2,P_3,P_4):=\frac{m\wedge l_1\wedge l_3}{m\wedge l_1\wedge l_2}\cdot\frac{m\wedge l_4\wedge l_2}{m\wedge l_4\wedge l_3}\]
as an extended real number. One can then verify that the cross ratio $(P_1,P_2,P_3,P_4)$ does not depend on the choice of $m$, $l_1$, $l_2$, $l_3$, $l_4$ or the choice of identification $f$. 

This definition of the cross ratio agrees with the classical notion of the cross ratio of four points on a line in the following way. By taking the dual, the four lines $P_1,\dots, P_4$ become four points $p_1,\dots,p_4\in(\Rbbb\Pbbb^2)^*$, and they lie in the projective line in $(\Rbbb\Pbbb^2)^*$ that corresponds to the point $[m]$ in $\Rbbb\Pbbb^2$. One can then check that $(P_1,P_2,P_3,P_4)$ is exactly the cross ratio of the four collinear points $p_1,\dots,p_4$.

\begin{proof}[Proof of (1) of Theorem \ref{main theorem} when $n=3$]
Choose orientations on $\eta$ and $\gamma$, and let $A$ and $B$ be elements in $\Gamma$ that correspond to $\eta$ and $\gamma$. Since $i(\eta,\gamma)\neq 0$, we can choose $A$ and $B$ so that if $a^+$, $b^+$ are the attracting fixed points and $a^-$, $b^-$ are the repelling fixed points of $A$ and $B$ respectively, then 
\[\ a^+,\ A\cdot b^+,\ b^+,\ a^-,\ b^-,\ A\cdot b^-\] 
lie in $\partial_\infty\Gamma$ in that cyclic order. By Lemma \ref{topological}, we see that 
\[a^+,\  A\cdot b^+,\  B\cdot a^+,\  b^+,\  a^-,\  b^-,\ B^{-1}\cdot a^+,\  A\cdot b^-\] 
lie in $\partial_\infty\Gamma$ in that cyclic order, because $A\cdot b^+$ and $A\cdot b^-$ are the attracting and repelling fixed points of $ABA^{-1}$ respectively.

Choose any $\rho\in Hit_3(S)$. For any non-identity element $X\in\Gamma$, let $\rho(X)^+$, $\rho(X)^0$, $\rho(X)^-$ be the three fixed points for $\rho(X)$, where $\rho(X)^+$ is attracting and $\rho(X)^-$ is repelling. Denote by $P_{\rho(X)}$ the line segment in $\Omega_\rho$ with endpoints $\rho(X)^+$ and $\rho(X)^-$. 

\begin{figure}
\includegraphics[scale=0.9]{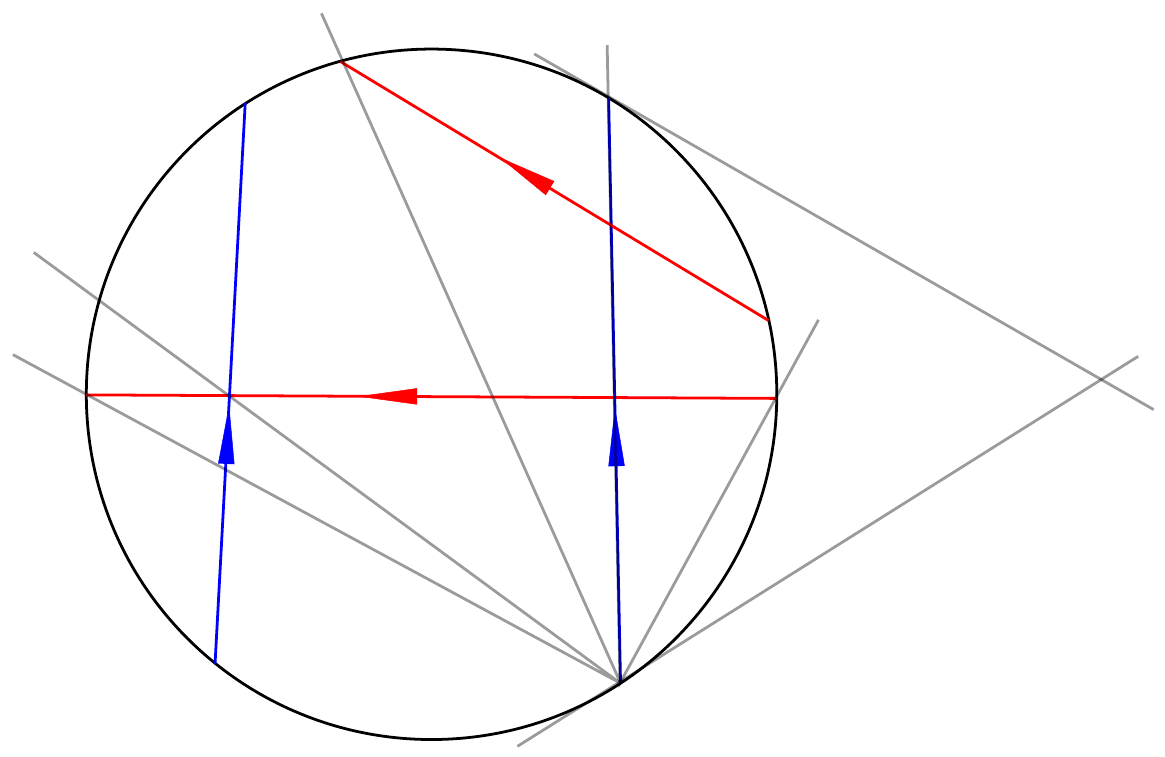}
\put (-440, 13){\makebox[0.7\textwidth][r]{$\Omega_\rho$ }}
\put (-497, 18){\makebox[0.7\textwidth][r]{$\rho(A)\cdot\rho(B)^-$ }}
\put (-531, 88){\makebox[0.7\textwidth][r]{$\rho(A)^+$ }}
\put (-488, 170){\makebox[0.7\textwidth][r]{$\rho(A)\cdot\rho(B)^+$ }}
\put (-447, 187){\makebox[0.7\textwidth][r]{$\rho(B)\cdot \rho(A)^+$ }}
\put (-365, 175){\makebox[0.7\textwidth][r]{$\rho(B)^+$ }}
\put (-293, 115){\makebox[0.7\textwidth][r]{$\rho(B)\cdot\rho(A)^-$ }}
\put (-319, 90){\makebox[0.7\textwidth][r]{$\rho(A)^-$ }}
\put (-233, 95){\makebox[0.7\textwidth][r]{$\rho(B)^0$ }}
\put (-362, 13){\makebox[0.7\textwidth][r]{$\rho(B)^-$ }}
\put (-448, 30){\makebox[0.7\textwidth][r]{$P_{\rho(ABA^{-1})}$ }}
\put (-455, 48){\makebox[0.7\textwidth][r]{$P_1$ }}
\put (-418, 50){\makebox[0.7\textwidth][r]{$P_2$ }}
\put (-418, 109){\makebox[0.7\textwidth][r]{$P_2'$ }}
\put (-438, 100){\makebox[0.7\textwidth][r]{$P_{\rho(A)}$ }}
\put (-398, 172){\makebox[0.7\textwidth][r]{$P_{\rho(BAB^{-1})}$ }}
\put (-366, 110){\makebox[0.7\textwidth][r]{$P_{\rho(B)}$ }}
\put (-366, 70){\makebox[0.7\textwidth][r]{$P_3$ }}
\put (-330, 42){\makebox[0.7\textwidth][r]{$P_3'$ }}
\caption{A schematic for the comparison between the cross ratios $(P_1,P_2,P_{\rho(B)},P_3)$ and $(P_1,P_2',P_{\rho(B)},P_3')$.}\label{Figure3}
\end{figure}
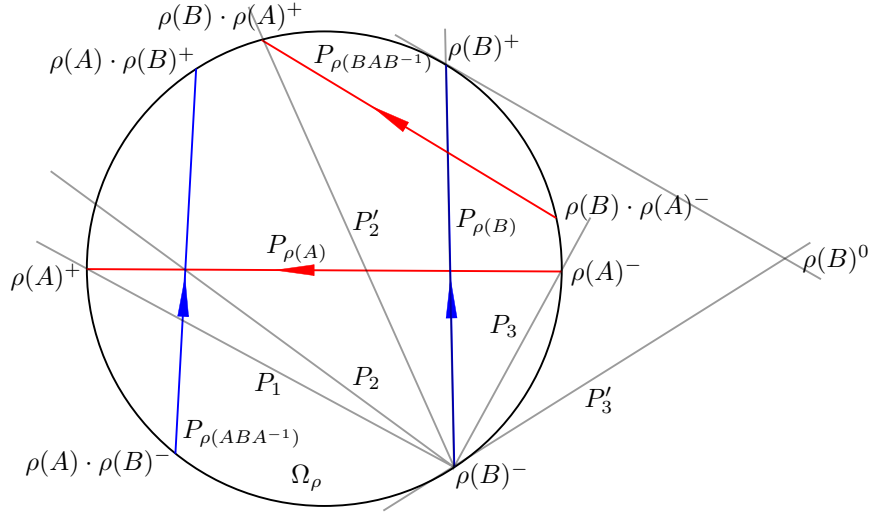

Now, let 
\begin{itemize}
\item $P_1$ be the line through $\rho(B)^-$ and $\rho(A)^+$, 
\item $P_2$ be the line through $\rho(B)^-$ and $P_{\rho(A)}\cap P_{\rho(ABA^{-1})}$, 
\item $P_3$ be the line through $\rho(B)^-$ and $\rho(A)^-$,
\item $P_2'$ be the line through $\rho(B)^-$ and $\rho(B)\cdot \rho(A)^+$,
\item $P_3'$ be the line through $\rho(B)^-$ and $\rho(B)^0$.
\end{itemize}
By using $\xi^{(1)}$ to identify $\partial_\infty\Gamma$ with $\partial\Omega_\rho$, we have that 
\[\rho(A)^+,\ \rho(A)\cdot\rho(B)^+,\ \rho(B)\cdot\rho(A)^+,\ \rho(B)^+,\ \rho(A)^-,\ \rho(B)^-\] 
lie in $\partial\Omega_\rho$ in that cyclic order (see Figure \ref{Figure3}.) It is a classically known property of the cross ratio (see Proposition \ref{useful cross ratio inequalities}) that
\[(P_1,P_2,P_{\rho(B)},P_3)> (P_1,P_2',P_{\rho(B)},P_3').\]

Let $0<\alpha_1<\alpha_2<\alpha_3$ be the eigenvalues of $\rho(A)$ and $0<\beta_1<\beta_2<\beta_3$ be the eigenvalues of $\rho(B)$. It is an easy cross ratio computation (see Lemma \ref{basic cross ratio} and Lemma \ref{cross ratio and length}) that
\[(P_1,P_2',P_{\rho(B)},P_3')=\frac{\beta_3}{\beta_3-\beta_2}.\] 
\[(P_1,P_2,P_{\rho(B)},P_3)=\frac{\alpha_3}{\alpha_1}.\] 

Hence, we have
\[\frac{\alpha_3}{\alpha_1}>\frac{\beta_3}{\beta_3-\beta_2},\]
which implies that
\[\frac{\beta_2}{\beta_3}< 1-\frac{\alpha_1}{\alpha_3}.\]
Similarly, by reversing the roles of $\rho(B)^-$ and $\rho(B)^+$, and using $\rho(B)^{-1}$ in place of $\rho(B)$, we can also show that 
\[\frac{\beta_1}{\beta_2}< 1-\frac{\alpha_1}{\alpha_3}.\]
Combining these inequalities gives 
\[\frac{\beta_1}{\beta_3}=\frac{\beta_1}{\beta_2}\cdot \frac{\beta_2}{\beta_3}<\bigg(1-\frac{\alpha_1}{\alpha_3}\bigg)^2,\]
which is equivalent to the inequality
\[\frac{\alpha_1}{\alpha_3}+\bigg(\frac{\beta_1}{\beta_3}\bigg)^{\frac{1}{2}}< 1.\]

Since $l_\rho(\eta)=\log(\frac{\alpha_3}{\alpha_1})$ and $l_\rho(\gamma)=\log(\frac{\beta_3}{\beta_1})$, (1) of Theorem \ref{main theorem} in the case when $n=3$ follows immediately.
\end{proof}

\subsection{Properties of Frenet curves of Hitchin representations.}\label{Hitchin properties}
Next, we want to generalize the proof given in Section \ref{n=3 case} to any Hitchin representation. We will devote this section to developing the tools needed to do so. In the rest of the paper, we use the same notation for points in $\Rbbb\Pbbb^{n-1}$ and for lines in $\Rbbb^n$. It should be clear from the context which we are referring to.

Denote by $\Fmc(\Rbbb^n)$ the space of complete flags in $\Rbbb^n$. Labourie \cite{Lab1} and Guichard \cite{Gui1} gave a beautiful characterization of representations in $Hit_n(S)$ as representations that admit an equivariant Frenet curve $\partial_\infty\Gamma\to\Fmc(\Rbbb^n)$. When $n=3$, the Frenet curve, post-composed with the projection from $\Fmc(\Rbbb^3)$ to $\Rbbb\Pbbb^2$, is exactly the map $\xi^{(1)}:\partial_\infty\Gamma\to\partial\Omega_\rho$ described in Section \ref{n=3 case}. This characterization will be the main tool we use to extend our proof in Section \ref{n=3 case} to the general case. 

We will start by first defining the Frenet property. 

\begin{notation}
let $\xi:S^1\to\Fmc(\Rbbb^n)$ be a continuous closed curve. For any $k=1,\dots,n-1$ and any point $x\in S^1$, let $\xi(x)^{(k)}:=\pi_k(\xi(x))$, where $\pi_k:\Fmc(\Rbbb^n)\to Gr(k,n)$ is the obvious projection.
\end{notation}

\begin{definition}
A closed curve $\xi:S^1\to\Fmc(\Rbbb^n)$ is \emph{Frenet} if the following hold:
\begin{enumerate}
\item Let $x_1,\dots,x_k$ be pairwise distinct points in $S^1$ and let $n_1,\dots,n_k$ be positive integers so that $\sum_{i=1}^k n_i=n$. Then 
\[\sum_{i=1}^k\xi(x_i)^{(n_i)}=\Rbbb^n.\]
\item Let $x_1,\dots,x_k$ be pairwise distinct points in $S^1$ and let $n_1,\dots,n_k$ be positive integers so that $m:=\sum_{i=1}^k n_i\leq n$. Then for any $x\in S^1$, 
\[\underset{x_i\neq x_j,\forall i\neq j}{\lim_{x_i\to x,\forall i}}\sum_{i=1}^k\xi(x_i)^{(n_i)}=\xi(x)^{(m)}.\]
\end{enumerate}
\end{definition}

The Frenet property ensures $\xi$ has good continuity properties and is ``maximally transverse". Combining the work of Labourie (Theorem 1.4 of \cite{Lab1}) and Guichard (Th\'eor\`eme 1 of \cite{Gui1}), one can characterize the representations in the $Hit_n(S)$ as those that preserve an equivariant Frenet curve.

\begin{thm}[Guichard, Labourie] 
A representation $\rho$ in the character variety 
\[Hom(\Gamma,PSL(n,\Rbbb))/ PSL(n,\Rbbb)\]
lies in $Hit_n(S)$ if and only if there exists a $\rho$-equivariant Frenet curve $\xi:\partial_\infty\Gamma\to\Fmc(\Rbbb^n)$. If $\xi$ exists, then it is unique.
\end{thm}

We will now prove several properties of these Frenet curves that will be needed. These are special cases of more general properties that appear in Section 2 of \cite{Zha1}. However, for the sake of completeness, we will reproduce the proofs.

\begin{lem}\label{important lemma}
Let $a,m_0,b,m_1,m_2,m_3$ be pairwise distinct points on $\partial_\infty\Gamma$ in that cyclic order and let $\rho\in Hit_n(S)$ with corresponding Frenet curve $\xi$. Also, let $P:=\Pbbb(\xi(a)^{(1)}+\xi(b)^{(1)})$. Then the following hold:
\begin{enumerate}
\item Let $k_0, k_1, k_2, k_3$, be non-negative integers that sum to $n-2$ and let $M:=\sum_{i=0}^3\xi(m_i)^{(k_i)}$. The map
\[f_M:\partial_\infty\Gamma\to P\]
given by
\[f_M:x\mapsto\begin{cases}
\Pbbb(\xi(x)^{(1)}+\sum_{i=0}^3\xi(m_i)^{(k_i)})\cap P&\text{ if }x\neq m_j\\
\Pbbb(\xi(m_j)^{(k_j+1)}+\sum_{i\neq j}\xi(m_i)^{(k_i)})\cap P&\text{ if }x=m_j\\
\end{cases}\]
is a homeomorphism with $f_M(a)=\xi(a)^{(1)}$ and $f_M(b)=\xi(b)^{(1)}$.
\item Let $k_0,k_1,k_2$ be non-negative integers that sum to $n-1$, and let $s$ be the closed subsegment of $\partial_\infty\Gamma$ with endpoints $a$, $b$ that does not contain $m_0$. Also, let $M:=\xi(m_0)^{(k_0)}$. Then there is some closed subsegment $\omega$ of $P$ with endpoints $\xi(a)^{(1)}$, $\xi(b)^{(1)}$ so that the map
\[g_M:s\to\omega\]
given by
\[g_M:x\mapsto\begin{cases}
\Pbbb(\xi(x)^{(k_2)}+\xi(m_1)^{(k_1)}+\xi(m_0)^{(k_0)})\cap P&\text{ if }x\neq m_1\\
\Pbbb(\xi(m_1)^{(k_1+k_2)}+\xi(m_0)^{(k_0)})\cap P&\text{ if }x=m_1\\
\end{cases}\]
is a homeomorphism with $g_M(a)=\xi(a)^{(1)}$ and $g_M(b)=\xi(b)^{(1)}$.
\end{enumerate}
\end{lem}

\begin{proof}
Before we start the proof, observe that for any non-negative integers $t_0,\dots,t_4$ so that $\sum_{i=0}^4t_i=n-1$, the intersection $\Pbbb(\xi(x)^{(t_4)}+\sum_{i=0}^3\xi(m_i)^{(t_i)})\cap P$ is a single point, otherwise $\Pbbb(\xi(a)^{(1)}+\xi(b)^{(1)})\subset\Pbbb(\xi(x)^{(t_4)}+\sum_{i=0}^3\xi(m_i)^{(t_i)})$, which contradicts the Frenet property of $\xi$.

Proof of (1). Since $\xi$ is Frenet, $f_M$ is continuous. Moreover, because the domain and target of $f_M$ are both topologically a circle, it is sufficient to show that $f_M$ is injective. Suppose for contradiction that there exist $x\neq x'$ such that $f_M(x)=f_M(x')$. We will assume that $x,x'\neq m_i$ for all $i=0,1,2,3$ as the other cases are similar. Then 
\begin{eqnarray*}
\sum_{i=0}^3\xi(m_i)^{(k_i)}+\xi(x)^{(1)}&=&\sum_{i=0}^3\xi(m_i)^{(k_i)}+f_M(x)\\
&=&\sum_{i=0}^3\xi(m_i)^{(k_i)}+f_M(x')\\
&=&\sum_{i=0}^3\xi(m_i)^{(k_i)}+\xi(x')^{(1)},
\end{eqnarray*}
which is impossible because $\xi$ is Frenet. The fact that $f_M(a)=\xi(a)^{(1)}$ and $f_M(b)=\xi(b)^{(1)}$ is easily verified.

Proof of (2). First, observe that $g_M$ viewed as a map from $s$ to $P$ is continuous. Also, for any $x$ in $s$, $g_M(x)=\xi(a)^{(1)}$ if and only if $x=a$ and $g_M(x)=\xi(b)^{(1)}$ if and only if $x=b$. This proves that the image of $g_M$ is a subsegment $\omega$ of $P$ with endpoints $\xi(a)^{(1)}$, $\xi(b)^{(1)}$. 

To finish the proof, we only need to show that $g_M$ is injective. Choose $x$, $x'$ in the interior of $s$ with $x\neq x'$, and assume without loss of generality that $a,x',x,b$ lie along $s$ in that order. Again, we assume that $x,x'\neq m_1$ as the other cases are similar. For any positive integer $i\leq k_2$, let 
\[M_i:=\xi(x)^{(i-1)}+\xi(x')^{(k_2-i)}+\xi(m_1)^{(k_1)}+\xi(m_0)^{(k_0)}.\] 
By (1), we know that $f_{M_i}(x)$ lies on $\omega$ strictly between $f_{M_i}(x')$ and $f_{M_i}(b)=\xi(b)^{(1)}$. Also, observe that $f_{M_i}(x)=f_{M_{i+1}}(x')$. This implies that $f_{M_{k_2}}(x)$ lies on $\omega$ strictly between $f_{M_1}(x')$ and $\xi(b)^{(1)}$. In particular, $g_M(x)=f_{M_{k_2}}(x)\neq f_{M_1}(x')=g_M(x')$, so $g_M$ is injective.
\end{proof}

In the proof of the $n=3$ case given in Section \ref{n=3 case}, the classical cross ratio in $\Rbbb\Pbbb^2$ was the main computational tool used to obtain our estimates. We will now define a generalization of the cross ratio for $\Rbbb\Pbbb^{n-1}$.

\begin{definition}\label{cross ratio definition}
Let $P_1,\dots,P_4$ be four hyperplanes in $\Rbbb^n$ that intersect along a $(n-2)$-dimensional subspace $M=\Span\{m_1,\dots,m_{n-2}\}\subset\Rbbb^n$, so that no three of the four $P_i$ agree. For $i=1,\dots,4$, let $L_i=[l_i]$ be a line through the origin in $P_i$ that does not lie in $M$. Define the \emph{cross ratio} by
\[(P_1,P_2,P_3,P_4):=\frac{m_1\wedge\dots\wedge m_{n-2}\wedge l_1\wedge l_3\cdot m_1\wedge\dots\wedge m_{n-2}\wedge l_4\wedge l_2}{m_1\wedge\dots\wedge m_{n-2}\wedge l_1\wedge l_2\cdot m_1\wedge\dots\wedge m_{n-2}\wedge l_4\wedge l_3}.\]
\end{definition}

In the above definition, choose an identification between $\bigwedge^n(\Rbbb^n)$ and $\Rbbb$ to evaluate the fraction on the right as a real number. One can check that this number does not depend on the identification chosen, the choice of basis $\{m_1,\dots,m_{n-2}\}$ for $M$, the choice of $L_i$ in $P_i$, or the choice of representatives $l_i$ for $L_i$. When convenient, we sometimes use the notation 
\[(L_1,L_2,L_3,L_4)_M:=(P_1,P_2,P_3,P_4).\] 
Also, at times, in our notation for the cross ratio, we replace the subspaces $L_i$, $P_i$ and $M$ of $\Rbbb^n$  with their projectivizations. As with the $n=3$ case, this definition of the cross ratio agrees with the classical cross ratio of four points along a projective line in $(\Rbbb\Pbbb^{n-1})^*$.

The following two lemmas summarizes some basic properties of this cross ratio.

\begin{lem}\label{basic cross ratio}
Let $L_1,\dots,L_5$ be pairwise distinct lines in $\Rbbb^n$ through $0$ and let $M$, $M'$ be $(n-2)$-dimensional subspaces of $\Rbbb^n$ not containing $L_i$ for any $i=1,\dots,5$, so that no three of the five $M+L_i$ agree and no three of the five $M'+L_i$ agree.
\begin{enumerate}
\item For any $X\in PSL(n,\Rbbb)$, $(X\cdot L_1,\dots, X\cdot L_4)_{X\cdot M}=(L_1,\dots L_4)_M$.
\item Suppose $L_1,L_2,L_3,L_4$ lie in a plane. Then 
\[(L_1,L_2,L_3,L_4)_M=(L_1,L_2,L_3,L_4)_{M'}.\]
\item $(L_1,L_2,L_3,L_4)_M=(L_4,L_3,L_2,L_1)_M$.
\item $(L_1,L_2,L_3,L_5)_M\cdot(L_1,L_3,L_4,L_5)_M=(L_1,L_2,L_4,L_5)_M$.
\item $(L_1,L_2,L_3,L_4)_M\cdot (L_1,L_3,L_2,L_4)_M=1$.
\item $(L_1,L_2,L_3,L_4)_M=1-(L_1,L_2,L_4,L_3)_M$.
\end{enumerate}
\end{lem}

\begin{proof}
(1), (3), (4) and (5) follow immediately from the definition of the cross ratio. To prove (2), observe that there is a projective transformation $X$ that fixes $L_1$, $L_2$, $L_3$ and maps $M$ to $M'$. Since $L_4$ lies in the plane containing $L_1$, $L_2$ and $L_3$, $X$ must also fix $L_4$. This allows us to use (1) to get (2). 

To prove (6), assume that $M+L_1,\dots,M+L_4$ are pairwise distinct; the other cases are similar. Choose a basis $e_1,\dots,e_n$ for $\Rbbb^n$ so that 
\[M=\Span\{e_1,\dots,e_{n-2}\},\ L_1=[e_{n-1}],\ L_4=[e_n],\ L_2=[\sum_{i=1}^ne_i],\ L_3=[\sum_{i=1}^n\alpha_ie_i]\]
for some real numbers $\alpha_1,\dots,\alpha_n$. The assumption that $M+L_1,\dots,M+L_4$ are pairwise distinct implies that $\alpha_{n-1}$ and $\alpha_n$ are non-zero real numbers. One can then easily compute that 
\[(L_1,L_2,L_3,L_4)_M=\frac{\alpha_n}{\alpha_{n-1}}\text{ and }(L_1,L_2,L_4,L_3)_M=\frac{\alpha_{n-1}-\alpha_n}{\alpha_{n-1}}.\]
\end{proof}

In view of (2) of Lemma \ref{basic cross ratio}, we will denote $(L_1,L_2,L_3,L_4)_M$ by $(L_1,L_2,L_3,L_4)$ in the case when $L_1$, $L_2$, $L_3$, $L_4$ lie in the same plane.

\begin{lem}\label{cross ratio and length}
Let $X\in PSL(n,\Rbbb)$ be diagonalizable with $n$ real eigenvalues $\lambda_1,\dots,\lambda_n$ of pairwise distinct moduli, so that $|\lambda_1|<\dots<|\lambda_n|$. Let $L_i$ and $L_j$ be fixed lines through the origin in $\Rbbb^n$ corresponding to the eigenvalues $\lambda_i$ and $\lambda_j$ respectively, with $i<j$, and let $L$ be a line through the origin in the plane $L_i+L_j$ so that $L_i\neq L\neq L_j$. Then
\[(L_i,L,X\cdot L,L_j)=\frac{\lambda_j}{\lambda_i}.\]
\end{lem}

\begin{proof}
Choose a basis $e_1,\dots,e_n$ for $\Rbbb^n$ so that $[e_k]$ is a fixed line through the origin of $\rho(X)$ corresponding to the eigenvalue $\lambda_k$.  In this basis, $\rho(X)$ is the diagonal matrix $[x_{u,v}]$, where
\[x_{u,v}=\left\{\begin{array}{lll}
0 & \text{if} &u\neq v\\
\lambda_u &\text{if} &u=v.
\end{array}\right.\]
Let $M$ be the $n-2$ dimensional subspace $\Span\{e_1,\dots,\hat{e}_i,\dots,\hat{e}_j,\dots,e_n\}$ of $\Rbbb^n$. Via a projective transformation that fixes $e_1,\dots,e_n$, we can assume $L=[e_i+e_j]$. The lemma follows from an easy computation using the cross ratio definition.
\end{proof}

The next task is to understand how the cross ratio interacts with Frenet curves. 

\begin{prop}\label{useful cross ratio inequalities}
Let $\rho\in Hit_n(S)$ and let $\xi$ be the corresponding Frenet curve. Also, let $a,b,c,m_0,d,m_1$ be pairwise distinct points along $\partial_\infty\Gamma$, in that cyclic order, and let $k_0$, $k_1$ be non-negative integers that sum to $n-2$. For any $x\in\partial_\infty\Gamma$, define
\[P_x=\left\{\begin{array}{lll} 
\xi(x)^{(1)}+\xi(m_0)^{(k_0)}+\xi(m_1)^{(k_1)} &\text{if} & x\neq m_0,m_1\\
\xi(m_i)^{(k_i+1)}+\xi(m_{1-i})^{(k_{1-i})}& \text{if} &x=m_i\\
\end{array}\right.\]
Then the following hold:
\begin{enumerate}
\item $(P_a,P_b,P_{m_0},P_d)>(P_a,P_b,P_{m_0},P_{m_1}).$
\item $(P_a,P_b,P_{m_0},P_d)>(P_a,P_c,P_{m_0},P_d).$
\end{enumerate}
\end{prop}

\begin{proof}
We will only show the proof of (1); the same proof together with the Lemma \ref{basic cross ratio} gives (2). Let
\begin{eqnarray*}
L_{m_0}&=&P_{m_0}\cap\bigg(\xi(a)^{(1)}+\xi(b)^{(1)}\bigg),\\
L_{m_1}&=&P_{m_1}\cap\bigg(\xi(a)^{(1)}+\xi(b)^{(1)}\bigg),\\
L_d&=&P_d\cap\bigg(\xi(a)^{(1)}+\xi(b)^{(1)}\bigg).
\end{eqnarray*}

Choose vectors $l_{m_0},l_{m_1},l_a,l_b,l_d$ in $\Rbbb^n$ such that 
\[[l_{m_0}]=L_{m_0},\ [l_{m_1}]=L_{m_1},\ [l_a]=\xi(a)^{(1)},\ [l_b]=\xi(b)^{(1)},\ [l_d]=L_d.\] 
By (1) of Lemma \ref{important lemma}, we can ensure, by replacing each $l_i$ with $-l_i$ if necessary, that 
\begin{eqnarray*}
l_{m_0}&=&\alpha l_a+(1-\alpha)l_b,\\ 
l_d&=&\beta l_a+(1-\beta)l_b,\\ 
l_{m_1}&=&\gamma l_a+(1-\gamma)l_b
\end{eqnarray*}
for $0<\alpha<\beta<\gamma<1$. Then one can compute
\begin{eqnarray*}
(P_a,P_b,P_{m_0},P_d)&=&\frac{1-\alpha}{1-\frac{\alpha}{\beta}}\\
&>&\frac{1-\alpha}{1-\frac{\alpha}{\gamma}}\\
&=&(P_a,P_b,P_{m_0},P_{m_1})
\end{eqnarray*}
\end{proof}

\subsection{Proof in the $PSL(n,\Rbbb)$ case.}\label{general}
We will now use the technical facts established in Section \ref{Hitchin properties} to prove Theorem \ref{main theorem}. For the rest of this section, fix $\rho\in Hit_n(S)$ and let $\xi$ be its corresponding Frenet curve. By a theorem of Labourie (Theorem 1.5 of \cite{Lab1}), we know that for every non-identity element $X\in\Gamma$, $\rho(X)\in PSL(n,\Rbbb)$ has a lift to $SL(n,\Rbbb)$ that is diagonalizable with positive pairwise distinct eigenvalues. These eigenvalues will also be referred to as the eigenvalues of $\rho(X)$. 

The next lemma is the main computation in the proof of Theorem \ref{main theorem}.

\begin{lem}\label{main computation}
Let $B$ be a non-identity element in $\Gamma$ and let $b^-$, $b^+$ be the repelling and attracting fixed points of $B$ respectively. Pick $k=0,\dots,n-2$, and for any $x\in\partial_\infty\Gamma$, define
\[P_x=P^{(k)}_x:=\left\{\begin{array}{lll}
\xi(b^-)^{(k)}+\xi(b^+)^{(n-k-2)}+\xi(x)^{(1)}&\text{if}&x\neq b^-,b^+\\
\xi(b^-)^{(k+1)}+\xi(b^+)^{(n-k-2)}&\text{if}& x=b^-\\
\xi(b^-)^{(k)}+\xi(b^+)^{(n-k-1)}&\text{if}& x=b^+\\
\end{array}\right.\] 
Suppose that $x_1$, $x_2$, $x_3$ are points in $\partial_\infty\Gamma$ so that 
\[x_1,\  x_2,\  B\cdot x_1,\  b^+,\  x_3,\  b^-\] 
lie on $\partial_\infty\Gamma$, in that cyclic order. Then
\[(P_{x_1},P_{x_2}, P_{b^+},P_{x_3})>\frac{\beta_{k+2}}{\beta_{k+2}-\beta_{k+1}},\]
where $0<\beta_1<\dots<\beta_n$ are the eigenvalues of $\rho(B)$.
\end{lem}

\begin{proof}
By Proposition \ref{useful cross ratio inequalities} and (5), (6) of Lemma \ref{basic cross ratio}, we have
\begin{eqnarray}
(P_{x_1},P_{x_2},P_{b^+},P_{x_3})&>&(P_{x_1},P_{B\cdot x_1},P_{b^+},P_{b^-})\nonumber\\
&=&\frac{1}{(P_{x_1},P_{b^+},P_{B\cdot x_1},P_{b^-})}\nonumber\\
&=&\frac{1}{1-(P_{b^+},P_{x_1},P_{B\cdot x_1},P_{b^-})}\label{B inequality}\\
&=&\frac{(P_{b^+},P_{B\cdot x_1},P_{x_1},P_{b^-})}{(P_{b^+},P_{B\cdot x_1},P_{x_1},P_{b^-})-1}\nonumber
\end{eqnarray}

Note that for all $j=1,\dots,n$, 
\[L_j:=\xi(b^-)^{(j)}\cap\xi(b^+)^{(n-j+1)}\]
is the fixed line through the origin in $\Rbbb^n$ of $\rho(B)$ corresponding to the eigenvalue $\beta_j$. Also, observe that $P_{b^+}$ and $P_{b^-}$ intersect the plane $\xi(b^-)^{(k+2)}\cap\xi(b^+)^{(n-k)}$ at $L_{k+2}$ and $L_{k+1}$ respectively. Let 
\[L:=P_{x_1}\cap(\xi(b^-)^{(k+2)}\cap\xi(b^+)^{(n-k)}),\] 
and it is clear that $P_{B\cdot x_1}\cap(\xi(b^-)^{(k+2)}\cap\xi(b^+)^{(n-k)})=\rho(B)\cdot L$. Thus, we can use Lemma \ref{cross ratio and length}, to conclude that 
\[(P_{b^+},P_{B\cdot x_1},P_{x_1},P_{b^-})=(L_{k+2},\rho(B)\cdot L,L,L_{k+1})=\frac{\beta_{k+2}}{\beta_{k+1}}.\]
Combining this with inequality (\ref{B inequality}) proves the lemma.
\end{proof}

Applying Lemma \ref{main computation} to our setting, we obtain the following proposition.

\begin{prop}\label{main proposition}
Let $A$, $B$ be elements in $\Gamma$ so that $a^+$, $b^+$, $a^-$, $b^-$ lie in $\partial_\infty\Gamma$ in that cyclic order. Here, $a^+$, $b^+$ are the attracting fixed points and $a^-$, $b^-$ are the repelling fixed points for $A$, $B$ respectively. Also, let $\alpha_1<\dots<\alpha_n$, $\beta_1<\dots<\beta_n$ be the eigenvalues of $\rho(A)$, $\rho(B)$ respectively. For every $k=0,\dots,n-2$, the following hold:
\begin{enumerate}
\item \[\frac{\alpha_n}{\alpha_1}>\frac{\beta_{k+2}}{\beta_{k+2}-\beta_{k+1}}.\]
\item Let $\eta$, $\gamma$ be closed curves in $S$ corresponding to $A$ and $B$ respectively. If $\gamma$ is simple, then  
\[\frac{\alpha_n}{\alpha_1}> \bigg(\frac{\beta_{k+2}}{\beta_{k+2}-\beta_{k+1}}\bigg)^u\cdot\bigg(\frac{\beta_{n-k}}{\beta_{n-k}-\beta_{n-k-1}}\bigg)^{i(\eta,\gamma)-u}\]
for some non-negative integer $u\leq i(\eta,\gamma)$ that is independent of $k$.
\end{enumerate}
\end{prop}

\begin{proof}
Proof of (1). Let $\omega$ be the subsegment of $\Pbbb(\xi(a^+)^{(1)}+\xi(a^-)^{(1)})$ with endpoints $\xi(a^+)^{(1)}$, $\xi(a^-)^{(1)}$ that has non-empty intersection with $\Pbbb(\xi(b^-)^{(k)}+\xi(b^+)^{(n-k-1)})$. Define
\begin{eqnarray*}
p&:=&\Pbbb(\xi(b^-)^{(k)}+\xi(b^+)^{(n-k-1)})\cap\omega,\\
q&:=&\Pbbb(\xi(b^-)^{(k)}+\xi(b^+)^{(n-k-2)}+\xi(A\cdot b^+)^{(1)})\cap\omega,
\end{eqnarray*}
and note that $q$ exists by (1) of Lemma \ref{important lemma}. Also, observe that
\[\rho(A)\cdot p=\Pbbb(\xi(A\cdot b^-)^{(k)}+\xi(A\cdot b^+)^{(n-k-1)})\cap\omega,\]
so (2) of Lemma \ref{important lemma} implies that $\rho(A)\cdot p$ lies between $\xi(a^+)^{(1)}$ and $q$ in $\omega$. Lemma \ref{basic cross ratio}, Lemma \ref{cross ratio and length} and Proposition \ref{useful cross ratio inequalities} together then allow us to conclude that
\begin{align*}
\frac{\alpha_n}{\alpha_1}&=(\xi(a^+)^{(1)},\rho(A)\cdot p,p,\xi(a^-)^{(1)})\\
&> (\xi(a^+)^{(1)},q,p,\xi(a^-)^{(1)})\\
&=(P_{a^+},P_{A\cdot b^+},P_{b^+},P_{a^-})
\end{align*}
where 
\[P_x=P^{(k)}_x:=\left\{\begin{array}{lll}
\xi(x)^{(1)}+\xi(b^-)^{(k)}+\xi(b^+)^{(n-k-2)}&\text{if}&x\neq b^-,b^+\\
\xi(b^-)^{(k+1)}+\xi(b^+)^{(n-k-2)}&\text{if}&x=b^-\\
\xi(b^-)^{(k)}+\xi(b^+)^{(n-k-1)}&\text{if}&x=b^+.\\
\end{array}\right. \]
By Lemma \ref{topological}, we know that $a^+$, $A\cdot b^+$, $B\cdot a^+$, $b^+$, $a^-$, $b^-$ lie along $\partial_\infty\Gamma$ in that cyclic order. This allows us to apply Lemma \ref{main computation} with $x_1$, $x_2$, $x_3$ as $a^+$, $A\cdot b^+$, $a^-$ respectively to obtain the desired inequality.

Proof of (2). Let $r^-$, $r^+$ be the closed subsegments of $\partial_\infty\Gamma$ with endpoints $a^-$ and $a^+$, so that $b^-$ and $A\cdot b^-$ lie in $r^-$, while $b^+$ and $A\cdot b^+$ lie in $r^+$. Orient both $r^-$ and $r^+$ from $a^-$ to $a^+$. Define $\Bmc$ to be the set of unordered pairs $\{b'^+,b'^-\}$ in the $\Gamma$-orbit of $\{b^+,b^-\}$ so that $b'^+$ lies in $r^+$ between $b^+$ and $A\cdot b^+$, while $b'^-$ lies in $r^-$ between $b^-$ and $A\cdot b^-$.

Every pair in $\Bmc$ is the set of attracting and repelling fixed points for some $B'$ in $\Gamma$ that is conjugate to $B$. Since $\eta$ is simple, we know that for every $\{b'^+,b'^-\}$ and $\{b''^+,b''^-\}$ in $\Bmc$, $b'^+$ precedes $b''^+$ (in the orientation on $r^+$) if and only if $b'^-$ precedes $b''^-$ (in the orientation of $r^-$). The orientation on $r^-$ and $r^+$ thus induce an ordering on $\Bmc$. Also, observe that $|\Bmc|=i(\eta,\gamma)+1$, so we can label the pairs in $\Bmc$ according to the order, i.e.
\[\Bmc=\bigg\{\{b_1^+,b_1^-\},\dots,\{b_{m+1}^+,b_{m+1}^-\}\bigg\},\]
where $b_1^+=b^+$, $b_1^-=b^-$, $b_{m+1}^+=A\cdot b^+$, $b_{m+1}^-=A\cdot b^-$, and $m=i(\eta,\gamma)$. 

For each $i$, let $B_i$ be the element in $\Gamma$ that is conjugate to either $B$ or $B^{-1}$ so that its attracting and repelling fixed points are $b_i^+$ and $b_i^-$ respectively. By Lemma \ref{topological}, $a^+$, $b_{i+1}^+$, $B_i\cdot a^+$, $b_i^+$, $a^-$, $b_i^-$ lie along $\partial_\infty\Gamma$ in that cyclic order, so we can apply Proposition \ref{main computation} with $x_1$, $x_2$, $x_3$ as $a^+$, $b_{i+1}^+$, $a^-$ respectively to conclude that 
\begin{equation}\label{B}
(P_{a^+,i},P_{b_{i+1}^+,i},P_{b_i^+,i},P_{a^-,i})>\frac{\beta_{k+2}}{\beta_{k+2}-\beta_{k+1}}\end{equation}
if $B_i$ is conjugate to $B$, and 
\begin{equation}\label{B inverse}
(P_{a^+,i},P_{b_{i+1}^+,i},P_{b_i^+,i},P_{a^-,i})>\frac{\beta_{n-k}}{\beta_{n-k}-\beta_{n-k-1}}\end{equation}
if $B_i$ is conjugate to $B^{-1}$, where
\[P_{x,i}=P^{(k)}_{x_i}:=\left\{\begin{array}{lll}
\xi(x)^{(1)}+\xi(b_i^-)^{(k)}+\xi(b_i^+)^{(n-k-2)}&\text{if}&x\neq b_i^-,b_i^+\\
\xi(b_i^-)^{(k+1)}+\xi(b_i^+)^{(n-k-2)}&\text{if}&x=b_i^-\\
\xi(b_i^-)^{(k)}+\xi(b_i^+)^{(n-k-1)}&\text{if}&x=b_i^+.
\end{array}\right. \]

Fix any $k=0,\dots,n-2$, and let $\omega$ be the subsegment of $\Pbbb(\xi(a^+)^{(1)}+\xi(a^-)^{(1)})$ with endpoints $\xi(a^+)^{(1)}$, $\xi(a^-)^{(1)}$ that has non-empty intersection with $\Pbbb(\xi(b_1^-)^{(k)}+\xi(b_1^+)^{(n-k-1)})$. For $i=1,\dots,m+1$, define
\[p_i:=\Pbbb(\xi(b_i^-)^{(k)}+\xi(b_i^+)^{(n-k-1)})\cap\omega,\]
and for $i=2,\dots,m+1$, define
\[q_i:=\Pbbb(\xi(b_{i-1}^-)^{(k)}+\xi(b_{i-1}^+)^{(n-k-2)}+\xi(b_i^+)^{(1)})\cap\omega.\]

Observe that (2) of Lemma \ref{important lemma} implies that $p_i$ and $q_i$ are well-defined, and that
$\xi(a^-)^{(1)},p_1,q_1,p_2,q_2,\dots,p_m,q_m,p_{m+1},\xi(a^+)^{(1)}$ lie in $\omega$ in that order. Hence, by similar arguments as those used in the proof of (1), we have
\begin{eqnarray*}
(\xi(a^+)^{(1)},p_{i+1},p_i,\xi(a^-)^{(1)})&>& (\xi(a^+)^{(1)},q_{i+1},p_i,\xi(a^-)^{(1)})\\
&=&(P_{a^+,i},P_{b_{i+1}^+,i},P_{b_i^+,i},P_{a^-,i})
\end{eqnarray*}
We can then use Lemma \ref{cross ratio and length} and Lemma \ref{basic cross ratio} to obtain
\begin{eqnarray}\label{multi intersection}
\frac{\alpha_n}{\alpha_1}&=&(\xi(a^+)^{(1)},p_{m+1},p_1,\xi(a^-)^{(1)})\nonumber\\
&=&\prod_{i=1}^m(\xi(a^+)^{(1)},p_{i+1},p_i,\xi(a^-)^{(1)})\\
&>&\prod_{i=1}^m(P_{a^+,i},P_{b_{i+1}^+,i},P_{b_i^+,i},P_{a^-,i}).\nonumber
\end{eqnarray}

Let $\Bmc_+:=\{i:B_i\text{ is conjugate to }B\}$, $\Bmc_-:=\{i:B_i\text{ is conjugate to }B^{-1}\}$ and let $u:=|\Bmc_+|$. Then combining the inequalities (\ref{B}), (\ref{B inverse}) and (\ref{multi intersection}) yields
\begin{eqnarray*}
\frac{\alpha_n}{\alpha_1}&>&\prod_{i\in\Bmc_+}(P_{a^+,i},P_{b_{i+1}^+,i},P_{b_i^+,i},P_{a^-,i})\cdot \prod_{i\in\Bmc_-}(P_{a^+,i},P_{b_{i+1}^+,i},P_{b_i^+,i},P_{a^-,i})\\
&>&\bigg(\frac{\beta_{k+2}}{\beta_{k+2}-\beta_{k+1}}\bigg)^u\cdot\bigg(\frac{\beta_{n-k}}{\beta_{n-k}-\beta_{n-k-1}}\bigg)^{i(\eta,\gamma)-u}.
\end{eqnarray*}
\end{proof}

With Proposition \ref{main proposition}, we can now prove Theorem \ref{main theorem}.

\begin{proof}[Proof of Theorem \ref{main theorem}]
In this proof, we will use the same notation as we used in Proposition \ref{main proposition}.

Proof of (1). Choose orientations on $\gamma$ and $\eta$. The hypothesis on $\gamma$ and $\eta$ imply that there are group elements $A$, $B$ in $\Gamma$ corresponding to $\eta$, $\gamma$ respectively, so that 
\[a^+,\ b^+,\ a^-,\ b^-\] 
lie along $\partial_\infty\Gamma$ in that cyclic order. Let $0<\alpha_1<\dots<\alpha_n$ and $0<\beta_1<\dots<\beta_n$ be the eigenvalues of $\rho(A)$ and $\rho(B)$ respectively. By (1) of Proposition \ref{main proposition}, we know that for all $k=0,\dots,n-2$, 
\[\frac{\alpha_n}{\alpha_1}>\frac{\beta_{k+2}}{\beta_{k+2}-\beta_{k+1}},\]
which implies that 
\[\frac{\beta_{k+1}}{\beta_{k+2}}< 1-\frac{\alpha_1}{\alpha_n}.\]
Taking the product over all $k=0,\dots,n-2$, we get
\[\frac{\alpha_1}{\alpha_n}+\bigg(\frac{\beta_1}{\beta_n}\bigg)^{\frac{1}{n-1}}< 1.\]
Since $l_\rho(\eta)=\log(\frac{\alpha_n}{\alpha_1})$ and $l_\rho(\gamma)=\log(\frac{\beta_n}{\beta_1})$, the above inequality gives us (1).

Proof of (2). 
By Proposition 2.12, we know that there is some non-negative integer $u\leq i(\eta,\gamma)$ so that for any $k=0,\dots,n-2$, we have
\[\frac{\alpha_1}{\alpha_n} < \Big(1-\frac{\beta_{k+1}}{\beta_{k+2}}\Big)^u \Big(1-\frac{\beta_{n-k-1}}{\beta_{n-k}}\Big)^{i(\eta,\gamma)-u}.\]
In particular, we also have that for any $k=0,\dots,n-2$, 
\[\frac{\alpha_1}{\alpha_n} < \Big(1-\frac{\beta_{k+1}}{\beta_{k+2}}\Big)^{i(\eta,\gamma)-u} \Big(1-\frac{\beta_{n-k-1}}{\beta_{n-k}}\Big)^u,\]
so we can assume that
\[ \frac{\alpha_1}{\alpha_n} < \Big(1-\frac{\beta_{k+1}}{\beta_{k+2}}\Big)^u \Big(1-{\frac{\beta_{n-k-1}}{\beta_{n-k}}}\Big)^{v}. \]
for some non-negative integers $u,v$ so that $u \geq v$ and $u+v=i(\eta,\gamma)$. This implies that
\[\frac{\alpha_1}{\alpha_n} 
< \Big(1-{\frac{\beta_{k+1}}{\beta_{k+2}}}\Big)^u \Big(1-{\frac{\beta_{1}}{\beta_{n}}}\Big)^{v},\]
which we can rewrite as
\[\frac{\beta_{k+1}}{\beta_{k+2}} < 1 - \frac{ (\frac{\alpha_{1}}{\alpha_{n}})^{\frac{1}{u}} }{ ( 1 - \frac{\beta_1}{\beta_n} )^{\frac{v}{u}}}.\]
By taking the product of the above inequality over $k=0,\dots,n-2$, we have
\[{\Big( \frac{\beta_{1}}{\beta_{n}} \Big)^{\frac{1}{n-1}}} < 1 - \frac{ (\frac{\alpha_{1}}{\alpha_{n}})^{\frac{1}{u}} }{ ( 1 - \frac{\beta_1}{\beta_n} )^{\frac{v}{u}} },\]
which can be rewritten as
\[ \frac{\alpha_{1}}{\alpha_{n}} <  \Big( 1 - {\Big(\frac{\beta_1}{\beta_n}\Big)^{\frac{1}{n-1}}} \Big)^{u}  \Big( 1 - \frac{\beta_1}{\beta_n} \Big)^{v},\]
from which (2) follows.

Proof of (3). Choose an orientation on $\eta$. Since $\eta$ is non-simple, there are group elements $A, B$ corresponding to $\eta$ so that 
\[a^+,\ b^+,\ a^-,\ b^-\] 
lie along $\partial_\infty\Gamma$ in that cyclic order. Let $0<\alpha_1<\dots<\alpha_n$ and $0<\beta_1<\dots<\beta_n$ be the eigenvalues of $\rho(A)$ and $\rho(B)$ respectively. Note that $\rho(B)$ is either conjugate to $\rho(A)$ or $\rho(A)^{-1}$, so $\frac{\beta_n}{\beta_1}=\frac{\alpha_n}{\alpha_1}$. Hence, the same computation as the proof of (1) then yields the inequality
\[\bigg(\frac{\alpha_1}{\alpha_n}\bigg)^{\frac{1}{n-1}}+\frac{\alpha_1}{\alpha_n}< 1,\]
which is equivalent to
\begin{equation}\label{inequality}
\bigg(1-\frac{\alpha_1}{\alpha_n}\bigg)^{n-1}-\frac{\alpha_1}{\alpha_n}> 0.
\end{equation}

Consider the polynomial $P_n(x)=x^{n-1}+x-1$. Note that for $n\geq 2$, $P_n(x)$ is strictly increasing on the interval $[0,1]$, $P_n(0)=-1$ and $P_n(1)=1$. Hence, $P_n$ has a unique zero in the interval $(0,1)$, which we denote by $x_n$. It then follows that 
\[\{x\in[0,1]:P_n(x)>0\}=(x_n,1].\]

Also, observe that 
\[P_n\bigg(1-\frac{\alpha_1}{\alpha_n}\bigg)=\bigg(1-\frac{\alpha_1}{\alpha_n}\bigg)^{n-1}-\frac{\alpha_1}{\alpha_n}\]
and $0<1-\frac{\alpha_1}{\alpha_n}< 1$.
Since $\frac{\alpha_1}{\alpha_n}$ satisfies the inequality (\ref{inequality}), we have 
\[x_n< 1-\frac{\alpha_1}{\alpha_n}< 1.\]
This implies that 
\[l_\rho(\eta)=\log\bigg(\frac{\alpha_n}{\alpha_1}\bigg)>\delta_n:=-\log(1-x_n).\] 
\end{proof}

\section{Further remarks.}\label{Further remarks}

In this section, we prove some corollaries to Theorem \ref{main theorem}, show that it does not hold for quasi-Fuchsian representations, and perform a comparison with the classical collar lemma.

\subsection{Corollaries.}\label{corollaries}
In this subsection, we will mention some consequences to Theorem \ref{main theorem}. The first is a universal collar lemma that holds simultaneously for all Hitchin representations.

\begin{cor}\label{main corollary}
Let $S$ be a surface of genus $g\geq 2$, and let $\gamma$, $\eta$ be two non-contractible closed curves in $S$. Then for any $n\geq 2$ and any $\rho\in Hit_n(S)$, the following hold:
\begin{enumerate}
\item If $i(\eta,\gamma)\neq 0$, then
\[(\exp(l_\rho(\gamma))-1)(\exp(l_\rho(\eta))-1)> 1.\]
\item If $i(\eta,\gamma)\neq 0$ and $\gamma$ is simple, then  
\[(\exp(l_\rho(\gamma))-1)\bigg(\exp\bigg(\frac{l_\rho(\eta)}{ i(\eta,\gamma)}\bigg)-1\bigg)> 1.\]
\item If $\eta$ is a non-simple closed curve, then
\[l_\rho(\eta)>\log(2).\]
\end{enumerate}
\end{cor}

\begin{proof}
Proof of (1). For all $n\geq 2$, (1) of Theorem \ref{main theorem} tells us that
\[\frac{1}{\exp(l_\rho(\eta))}< 1-\frac{1}{\exp(\frac{l_\rho(\gamma)}{n-1})}\leq 1-\frac{1}{\exp(l_\rho(\gamma))}.\]
By rearranging the terms in this inequality, we then have
\[(\exp(l_\rho(\eta))-1)(\exp(l_\rho(\gamma))-1)> 1.\]

The proof of (2) is very similar to the proof of (1), and (3) is obvious since $\{\delta_i\}_{i=2}^\infty$ is an increasing sequence and $\delta_2=\log(2)$.
\end{proof}

An easy corollary to Theorem \ref{main theorem} is a generalization of a direct consequence of the classical collar lemma in the case of $\Tmc(S)$.

\begin{cor}
For any $n\geq 2$ and any $\rho\in Hit_n(S)$, there are at most $3g-3$ closed curves in S of $\rho$-length at most $\delta_n$.
\end{cor}

In the case of $\Tmc(S)$, one can replace the number $\delta_2=\log(2)$ with $4\cdot\sinh^{-1}(1)$ (see Theorem 4.2.2 of \cite{Bus1}).

\begin{proof}
By (1) of Theorem \ref{main theorem}, if $\eta$ and $\gamma$ are closed curves in $S$ such that $i(\eta,\gamma)\neq 0$, then $l_\rho(\eta)$ and $l_\rho(\gamma)$ cannot both be smaller than $\delta_n$. Moreover, (3) of Theorem \ref{main theorem} tells us that any closed curve of $\rho$-length less than $\delta_n$ has to be simple. Thus, the set of closed curves of $\rho$-length less than $\delta_n$ has to be a pairwise disjoint collection of simple closed curves, so the size of this collection is at most $3g-3$. 
\end{proof}

Let $\widetilde{M}$ be the $SL(n,\Rbbb)$ symmetric space, and let $d_{\widetilde{M}}$ be the distance function given by the Riemannian metric on $\widetilde{M}$. It is a well-known that for any $Z\in\Gamma$, the translation length of $\rho(Z)$, $\inf\{d_{\widetilde{M}}(o,\rho(Z)\cdot o):o\in \widetilde{M}\}$, is
\[c_n\sqrt{\sum_{i=1}^n(\log\lambda_i)^2},\]
for some positive constant $c_n$ depending only on $n$. Here, $0<\lambda_1<\dots<\lambda_n$ are the eigenvalues of $\rho(Z)$. (See II.10 of Bridson-Haefliger \cite{BriHae1}.) For our purposes, we normalize the metric on $\widetilde{M}$ so that $c_n=\sqrt{2}$, i.e. so that the image of the totally geodesic embedding of $\Hbbb^2$ in $\widetilde{M}$ induced by $\iota_n: PSL(2,\Rbbb)\to PSL(n,\Rbbb)$ has sectional curvature $-\frac{6}{n(n-1)(n+1)}$. Then for any discrete, faithful representation $\rho:\Gamma\to PSL(n,\Rbbb)$, and for any rectifiable closed curve $\omega$ in $M:=\rho(\Gamma)\backslash\widetilde{M}$, let $l_M(\omega)$ be the length of $\omega$ in the Riemannian metric on $M$. 

In the case when $\rho\in Hit_n(S)$, we can use Theorem \ref{main theorem}, to obtain a relationship between the lengths of curves in the quotient locally symmetric space $M$.

\begin{cor}\label{symmetric space}
Let $\gamma$, $\eta$ be two non-contractible closed curves in $S$ and let $X$, $Y$ be elements in $\Gamma$ corresponding to $\gamma$, $\eta$ respectively. For any $\rho\in Hit_n(S)$, let $\gamma'$, $\eta'$ be two closed curves in $\rho(\Gamma)\backslash\widetilde{M}=:M$ that correspond to $X,Y\in\Gamma=\pi_1(M)$ respectively. Then the statements in Theorem \ref{main theorem} and Corollary \ref{main corollary intro} hold, with $l_\rho(\gamma)$ and $l_\rho(\eta)$ replaced with $l_M(\gamma')$ and $l_M(\eta')$ respectively.
\end{cor}

\begin{proof}
Pick any $Z\in\Gamma\setminus\{\id\}$, and let $\omega$ in $S$ and $\omega'$ in $M$ be closed curves corresponding to $Z$. Observe then that the translation length of $\rho(Z)$ in $\widetilde{M}$ is a lower bound for $l_M(\omega')$. 

Also, since 
\[ 2 \sum_{i=1}^n x_i^2 - (x_n - x_1)^2 = (x_1 + x_n)^2 + 2(x_2^2 + \dotsm + x_{n-1}^2) \geq 0,\]
we have that
\[ (x_n - x_1)^2 \leq 2 \sum_{i=1}^n x_i^2.\]
This allows us to compute
\[l_\rho(\omega) = \log\Big(\frac{\lambda_n}{\lambda_1}\Big) \leq \sqrt{2 \sum_{i=1}^n(\log\lambda_i)^2} \leq l_M(\omega'),\]
where $0<\lambda_1<\dots<\lambda_n$ are the eigenvalues of $\rho(Z)$.
\end{proof}

Note that in Corollary \ref{symmetric space}, the closed curves $\gamma'$ and $\eta'$ in the locally symmetric space do not necessarily intersect, even when $i(\gamma,\eta)\neq 0$. In particular, Corollary \ref{symmetric space} is not simply a quantitative version of the Margulis lemma on $PSL(n,\Rbbb)$. However, it does imply a quantitative version of the Margulis lemma in the following way. Let $f:S\to M:=\rho(\Gamma)\backslash\widetilde{M}$ be a $\pi_1$-injective map such that $f(\gamma)$ and $f(\eta)$ are rectifiable curves in the Riemannian metric on $M$. Then the statements in Theorem \ref{main theorem} and Corollary \ref{main corollary intro} hold, with $l_\rho(\gamma)$ and $l_\rho(\eta)$ replaced with $l_M(f(\gamma))$ and $l_M(f(\eta))$ respectively. In particular, we have a collar lemma for the image of the harmonic maps corresponding to Hitchin representations that were given by Corlette \cite{Cor1}, Labourie \cite{Lab2} or Eells-Sampson \cite{EelSam1}.

Theorem \ref{main theorem} and Corollary \ref{main corollary intro} also allow us to deduce consequence that are similar to Corollary \ref{symmetric space}, but with the Hilbert metric on the symmetric space instead of the Riemannian one. The symmetric space $\widetilde{M}$ can be given a Hilbert metric in the following way. Let $S(n,\Rbbb)$ be the space of symmetric $n$ by $n$ matrices with real entries and let $P(n,\Rbbb)$ be the set of positive-definite matrices in $S(n,\Rbbb)$. Let $\Pbbb(P)$ and $\Pbbb(S)$ be the projectivizations of $P(n,\Rbbb)$ and $S(n,\Rbbb)$ respectively, and observe that $\Pbbb(P)$ is a properly convex domain in $\Pbbb(S)\simeq\Rbbb\Pbbb^{N-1}$, where $N=\frac{n(n+1)}{2}$. This allows us to equip $\Pbbb(P)$ with a Hilbert metric.

Moreover, we can define a $PSL(n,\Rbbb)$-action on $\Pbbb(S)$ by $g\cdot A := gAg^T$ for any $g\in PSL(n,\Rbbb)$ and any $A\in\Pbbb(S)$. Note that this action preserves the projective structure on $\Pbbb(S)$, and also preserves $\Pbbb(P)$. In fact, $PSL(n,\Rbbb)$ acts transitively on $\Pbbb(P)$, and the stabilizer of the projective class of the identity matrix in $\Pbbb(P)$ is $PSO(n)$, so the symmetric space $\widetilde{M}$ can be identified with $\Pbbb(P)$. This equips $\widetilde{M}$ with a Hilbert metric.  Denote $\widetilde{M}$ equipped with the Hilbert metric by $\widetilde{M}'$, and for any discrete, faithful representation $\rho:\Gamma\to PSL(n,\Rbbb)$, let $l_{M'}$ be the length function on $M':=\rho(\Gamma)\backslash\widetilde{M}'$ induced by the Hilbert metric. Theorem \ref{main theorem} and Corollary \ref{main corollary intro} then also implies the following corollary.

\begin{cor}
Let $\gamma$, $\eta$ be two non-contractible closed curves in $S$ and let $X$, $Y$ be elements in $\Gamma$ corresponding to $\gamma$, $\eta$ respectively. For any $\rho\in Hit_n(S)$, let $\gamma'$, $\eta'$ be two closed curves in $M'$ that correspond to $X,Y\in\Gamma=\pi_1(M')$ respectively. Then the statements in Theorem \ref{main theorem} and Corollary \ref{main corollary intro} hold, with $l_\rho(\gamma)$ and $l_\rho(\eta)$ replaced with $\frac{1}{2}l_{M'}(\gamma')$ and $\frac{1}{2}l_{M'}(\eta')$ respectively.
\end{cor}

\begin{proof}
For any $Z\in\Gamma\setminus\{\id\}$, let $0<\lambda_1<\dots<\lambda_n$ be the eigenvalues of $\rho(Z)$. We can assume without loss of generality that $\rho(Z)$ is a diagonal. Let $E_{ij}$ be the $(n,n)$--matrix with $1$ at position $(i,j)$ and zero everywhere else, and let $B_{ij} = E_{ij}+E_{ji}$. Obviously, $\{ B_{ij} \}_{i \leq j}$ is a basis of $S(n,\mathbb{R})=\mathbb{R}^{N}$, and it is easy to verify that $\rho(Z)\cdot B_{ij} = \lambda_{i}\lambda_{j} B_{ij}$. That means $B_{ij}$ is an eigenvector of $\rho(Z)$ with eigenvalue $\lambda_{i}\lambda_{j}$. Consequently, the translation length of $\rho(Z)$ is 
\[\log\bigg(\frac{\lambda_n^2}{\lambda_1^2}\bigg) = 2 \log\bigg(\frac{\lambda_n }{ \lambda_1}\bigg)\] 
(see Proposition 2.1 in \cite{CooLonTil1}). The corollary then follows easily.
\end{proof}

The image of the irreducible representation $\iota_n:PSL(2,\Rbbb)\to PSL(n,\Rbbb)$ lies in a conjugate of the subgroup $PSO(k,k+1)\subset PSL(2k+1,\Rbbb)$ when $n=2k+1$, and a conjugate of $PSp(2k,\Rbbb)\subset PSL(2k,\Rbbb)$ when $n=2k$. Hence, we can define Hitchin components of the character varieties
\[Hom(\Gamma,PSO(k,k+1))/PSO(k,k+1),\ Hom(\Gamma,PSp(2k,\Rbbb))/PSp(2k,\Rbbb)\] 
in the same way as we did for $PSL(n,\Rbbb)$. Denote these Hitchin components by $Hit_n(S)'$. Note that $Hit_n(S)'$ can be naturally identified with a subset of $Hit_n(S)$. In the case when $\rho\in Hit_n(S)$ happens to be an element of $Hit_n(S)'$, we can strengthen (2) of Proposition \ref{main proposition}, which we state as the following corollary.

\begin{cor}
Let $A$, $B$ be elements in $\Gamma$ so that $a^+$, $b^+$, $a^-$, $b^-$ lie in $\partial_\infty\Gamma$ in that cyclic order. Here, $a^+$, $b^+$ are the attracting fixed points and $a^-$, $b^-$ are the repelling fixed points for $A$, $B$ respectively. Let $\rho\in Hit_n(S)'$ and let $\alpha_1<\dots<\alpha_n$, $\beta_1<\dots<\beta_n$ be the eigenvalues of $\rho(A)$, $\rho(B)$ respectively. Finally, let $\eta$ and $\gamma$ be closed curves on $S$ corresponding to $A$ and $B$ respectively. If $\gamma$ is a simple closed curve in $S$, then for every $k=0,\dots,n-2$, 
\[\frac{\alpha_n}{\alpha_1}> \bigg(\frac{\beta_{k+2}}{\beta_{k+2}-\beta_{k+1}}\bigg)^{i(\eta,\gamma)}\]
\end{cor} 

\begin{proof}
Since $\rho(B)$ is a diagonalizable element in $PSO(k,k+1)$ or $PSp(2k,\Rbbb)$, we see that $\beta_{k+1}=\frac{1}{\beta_{n-k}}$ for $k=0,\dots,n-1$. Apply this to (2) of Proposition \ref{main proposition}.
\end{proof}

From this corollary, the same proof as (2) of Theorem \ref{main theorem} allows us to obtain the following stronger inequality in the case when $\rho\in Hit_n(S)'$.

\begin{cor}
Let $\gamma$, $\eta$ be two non-contractible closed curves in $S$ so that $\gamma$ is simple and $i(\eta,\gamma)\neq 0$. Then for any $\rho\in Hit_n(S)'$, 
\[\frac{1}{\exp(l_\rho(\eta))}<\bigg(1-\frac{1}{\exp(\frac{l_\rho(\gamma)}{n-1})}\bigg)^{i(\eta,\gamma)}.\]
\end{cor}

Finally, combining (1) of Proposition \ref{main proposition} with some results proven in \cite{Zha1}, one obtains the following properness result.

\begin{cor}\label{proper}
Let $\Cmc:=\{\gamma_1,\dots,\gamma_k\}$ be a collection of closed curves in $S$ that contains a pants decomposition, so that the complement of $\Cmc$ in $S$ is a union of discs. Then the map
\begin{eqnarray*}
\Psi:Hit_n(S)&\to&\Rbbb^k\\
\rho&\mapsto&(l_\rho(\gamma_1),\dots,l_\rho(\gamma_k))
\end{eqnarray*}
is proper.
\end{cor}

In other words, whether or not a sequence $\{\rho_i\}$ divergences in the Hitchin component can be detected by the $\rho_i$-lengths of the curves in $\Cmc$. We postpone the proof of this corollary to Appendix \ref{appendix} because it uses some technical results from \cite{Zha1}. 

\subsection{Counter example for non-Hitchin representations.}\label{counter example}
Note that in our proof, we used very strongly that the representations we consider are in $Hit_n(S)$ because we used properties of the Frenet curve to obtain our estimates. In fact, the collar lemma is special to Hitchin representations, and does not hold even on the space of discrete and faithful representations from $\Gamma$ to $PSL(n,\Rbbb)$. 

To see this, consider the space of conjugacy classes of quasi-Fuchsian representations from $\Gamma$ to $PSL(2,\Cbbb)=PSO(3,1)^+\subset PSL(4,\Rbbb)$, which is the group of orientation preserving isometries of $\Hbbb^3$. These are discrete and faithful representations whose limit set in the Riemann sphere is a Jordan curve. It is well-known that each quasi-Fuchsian representation $\rho$ induces a convex cocompact hyperbolic structure on the three-manifold $S\times I$. Also, for any non-identity element $X$ in $\Gamma$, the closed geodesic $\gamma$ in $S\times I$ (equipped with the hyperbolic metric induced by $\rho$) corresponding to $X$ has $\rho$-length 
\[l_\rho(\gamma)=\log\bigg|\frac{\lambda_4}{\lambda_1}\bigg|,\] 
where $\lambda_4$ and $\lambda_1$ are the eigenvalues of $\rho(X)$ with largest and smallest modulus respectively (see Proposition 2.1 of Cooper-Long-Tillman \cite{CooLonTil1}).

It is a theorem of Bers (Theorem 1 of \cite{Ber1}) that the space of quasi-Fuchsian representations can be naturally identified with $\Tmc(S)\times \Tmc(\overline{S})$, where $\overline{S}$ is $S$ with the opposite orientation. For any quasi-Fuchsian representation $\rho$ let $(\rho^+,\rho^-)$ denote the pair of Fuchsian representations that corresponds to $\rho$, so that $\rho^+\in\Tmc(S)$ and $\rho^-\in\Tmc(\overline{S})$. Then for any closed non-contractible curve $\gamma$ in $S$, let $\gamma_\rho$ be the geodesic representative of $\gamma$ in the hyperbolic metric on $S\times I$ corresponding to $\rho$, and let $\gamma_{\rho^+}$ and $\gamma_{\rho^-}$ be the geodesic representatives of $\gamma$ in the hyperbolic metrics on $S$ and $\overline{S}$ corresponding to $\rho^+$ and $\rho^-$ respectively. By Theorem 3.1 of Epstein-Marden-Markovic \cite{EpsMarMar1} we know that
\[l_\rho(\gamma_\rho)\leq\min\{2\cdot l_{\rho^+}(\gamma_{\rho^+}),2\cdot l_{\rho^-}(\gamma_{\rho^-})\}.\]

For any pair of simple closed curves $\eta$ and $\gamma$ and for any $\epsilon>0$, let $\rho$ be a quasi-Fuchsian representation so that 
\[l_{\rho^+}(\gamma_{\rho^+})<\frac{\epsilon}{2}\ \text{ and }\ l_{\rho^-}(\eta_{\rho^-})<\frac{\epsilon}{2}.\]
Hence, $l_\rho(\gamma_\rho)$ and $l_\rho(\eta_\rho)$ are both smaller than $\epsilon$. This implies that the analog of Theorem \ref{main theorem} does not hold on the space of discrete and faithful, or even Anosov, representations from $\Gamma$ to $PSL(4,\Rbbb)$. (See Guichard-Wienhard \cite{GuiWie1} for more background on Anosov representations.)

\subsection{Comparison with the classical collar lemma.}\label{comparison}
Let $\rho$ be a representation in the Fuchsian locus of $Hit_n(S)$ and let $h$ be the corresponding Fuchsian representation in $\Tmc(S)$. Also, let $X$ be a non-identity element in $\Gamma$ and let $\gamma$ be a curve in $S$ corresponding to $X$. If $\lambda^{-1}$ and $\lambda$ are the two eigenvalues of $h(X)$, then $\lambda^{-n+1}$, $\lambda^{-n+3}$, $\dotsc$, $\lambda^{n-3}$, $\lambda^{n-1}$ are the $n$ eigenvalues of $\rho(X)$. Hence we can get the lengths
\[ l_h(\gamma)=2 \log(\lambda) \quad \textrm{and} \quad  l_{\rho}(\gamma)=2(n-1)\log(\lambda).\]

\begin{figure}[ht]
\centering
\includegraphics[width=4.95cm]{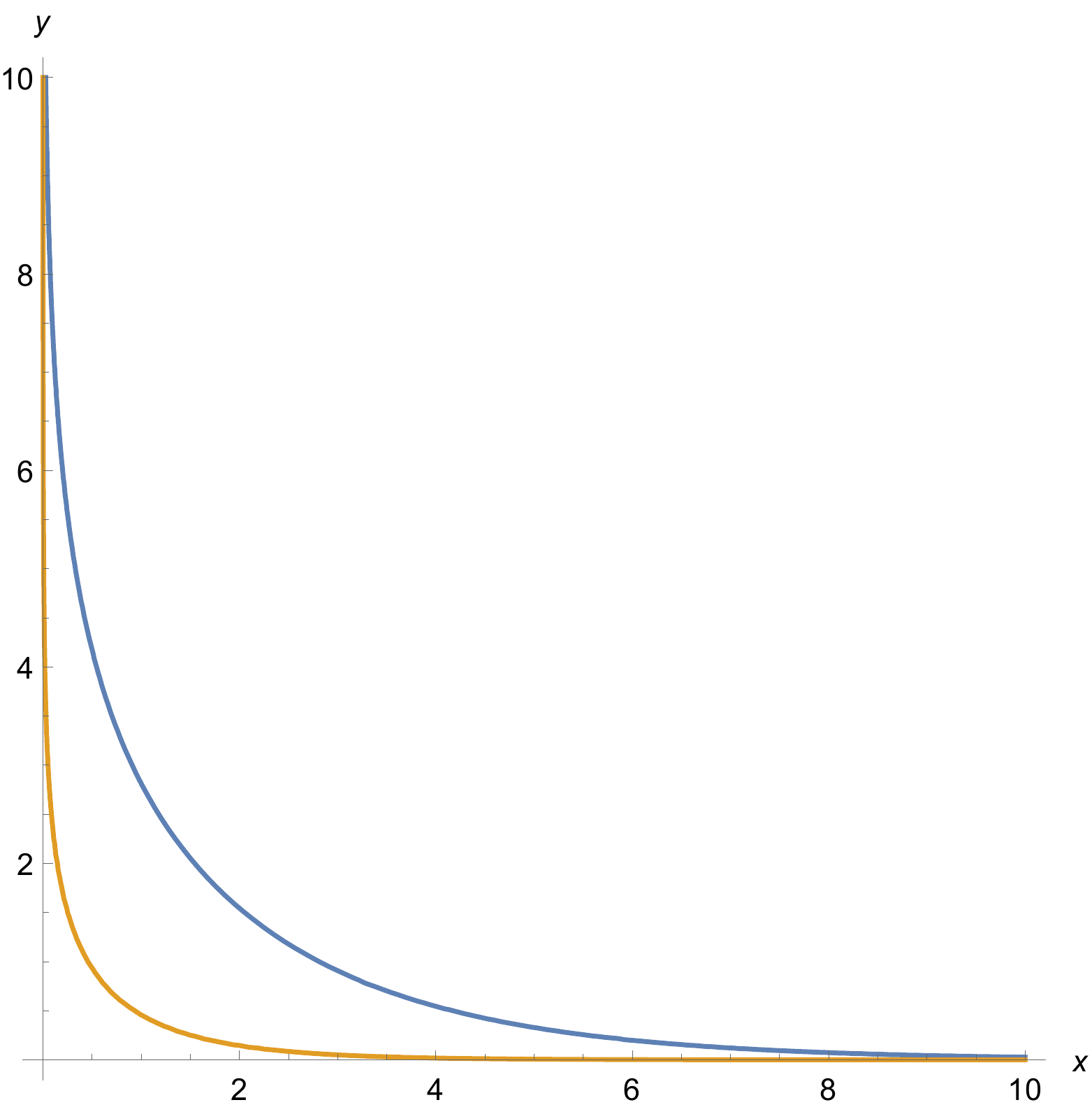}
\caption{The blue curve $\sinh (\frac{x}{2})\sinh (\frac{y}{2})=1$ and the orange curve $(e^{x} -1)(e^{y} -1)=1$}\label{graph}
\end{figure}

Since $h\in\Tmc(S)$, the classical collar lemma holds. In other words, for any pair of curves $\gamma$ and $\eta$ in $S$ such that $\eta$ is simple and $i(\eta,\gamma)>0$, we have (Corollary 4.1.2 of Buser \cite{Bus1})
\begin{equation}\label{classical collar lemma}
I_{\gamma,\eta}(h):=\sinh \Big(\frac{l_h(\gamma)}{2}  \Big) \sinh \Big(\frac{l_h(\eta)}{2}  \Big) >  1.
\end{equation}
This inequality is sharp, in the sense that for any $S$, there are curves $\eta$ and $\gamma$ in $S$ and a sequence of Fuchsian representations $\{h_i\}$ such that $I_{\gamma,\eta}(h_i)$ converges to $1$.
For more details, refer to Section 6 of Matelski \cite{Mat1}.

On the other hand, (1) of Corollary \ref{main corollary intro}, specialized to the $n=2$ case, is the inequality
\[(e^{l_h(\gamma)}-1)(e^{l_h(\eta)}-1)> 1.\]
This is weaker than the inequality (\ref{classical collar lemma}), because 
\[e^{x} -1 > \frac{e^{-\frac{1}{2}x}}{2}(e^x-1)=\sinh \Big(\frac{x}{2}\Big)\] 
for every $x > 0$ (see Figure \ref{graph}). Moreover, we are unable to show that the inequality (\ref{classical collar lemma}) fails in $Hit_n(S)$ for any $n>2$. This led us to make the following conjecture in an earlier version of this paper.

\begin{conj}
Let $\rho$ be a representation in $Hit_n(S)$. Then there is some representation $\rho'$ in the Fuchsian locus of $Hit_n(S)$ such that 
\[ l_{\rho}(\gamma) \geq l_{\rho'}(\gamma), \textrm{ for any } \gamma \in \Gamma.\]
\end{conj}

This conjecture implies that
\[\sinh \Big(\frac{l_\rho(\gamma)}{2(n-1)}  \Big) \sinh \Big(\frac{l_\rho(\eta)}{2(n-1)}  \Big) >  1\]
for any $\rho\in Hit_n(S)$, which is sharp on every $Hit_n(S)$ because it is sharp when restricted to the Fuchsian locus. Recently, Tholozan proved (Section 0.4 of \cite{Tho1}) that the conjecture holds in the case when $n=3$. Furthermore, F. Labourie disproved our conjecture in the case when $n\geq 4$. We will give his argument here.

For any closed curve $\gamma$ in $S$, let $L_\gamma : Hit_n(S) \rightarrow \mathbb{R}$ denote the map given by $ L_\gamma (\rho) = l_{\rho}(\gamma)$. As before, let $Hit_n(S)'$ be the $PSp(2k,\Rbbb)$ or $PSO(k,k+1)$ Hitchin components when $n=2k$ and $n=2k+1$ respectively, and recall that for all $\rho\in Hit_n(S)'$, $l_\rho(\gamma)=2\log|\lambda_n(\rho(X))|$, where $X\in\Gamma$ is a group element corresponding to $\gamma$ and $\lambda_n(\rho(X))$ is the eigenvalue of $\rho(X)$ with largest modulus. Proposition 10.3 of Bridgeman-Canary-Labourie-Sambarino \cite{BriCanLabSam1} then implies that for any $\rho\in Hit_n(S)'$, the set of differentials $\{dL_\gamma:\gamma\text{ a closed curve in }S\}$ generates the entire cotangent space of $Hit_n(S)'$ at $\rho$.

Observe that when $n\geq 4$, $Hit_n(S)'\subset Hit_n(S)$ properly contains the Fuchsian locus. Thus, to show that the conjecture is false, it is sufficient to show that it fails on $Hit_n(S)'$. Suppose for contradiction that the conjecture holds on $Hit_n(S)'$. Choose a point $\rho_0$ in the Fuchsian locus, and take a smooth path $\rho_t$, $t \in (-\epsilon, \epsilon)$ with $\epsilon > 0$, whose non-zero tangent vector $U \in T_{\rho_0}Hit_n(S)'$ is not tangential to the Fuchsian locus. Along the path, choose a sequence of representations $\{\rho_{t_i}\}_{i=1}^{\infty}$ which converges to $\rho_0$ as $i \rightarrow \infty$ so that $t_i >0$ for all $i$.

Since the conjecture holds, there exists the corresponding sequence of Fuchsian representations $\rho'_{t_i}$ such that $L_{\gamma}(\rho_{t_i})  \geq L_{\gamma}(\rho'_{t_i}) , \textrm{ for any } \gamma \in \Gamma$. Also, since $\rho_{t_i}$ converges to $\rho_0$, we see that $L_{\gamma}(\rho'_{t_i})$ is bounded for all $\gamma \in \Gamma$, so the sequence $\{\rho'_{t_i}\}_{i=1}^\infty$ lie in a compact subset of the Fuchsian locus. By picking subsequence, we can assume without loss of generality that $\{ \rho'_{t_i} \}_{i=1}^\infty$ converges to some $\rho'_{0}$ in the Fuchsian locus. The continuity of $L_\gamma$ then implies that $L_\gamma(\rho_0)\geq L_\gamma(\rho_0')$ for all $\gamma\in\Gamma$, so  $\rho_{0} = \rho'_{0}$ because both $\rho_{0}$ and $\rho'_{0}$ lie in the Fuchsian locus. 

Thus, the sequence $\{\rho_{t_i}'\}_{i=1}^\infty$ converges to $\rho_0$ as well. Choose a Riemannian metric on a neighborhood of $\rho_0$ in $Hit_n(S)$. By taking a further subsequence of $\{ \rho_{t_i} \}_{i=1}^\infty$, we can also assume that either one of the following cases hold:
\begin{enumerate}[(i)]
\item $\rho_{t_i}'=\rho_0$ for all $i$.
\item the unit vectors at $\rho_0$ that are tangential to the geodesic between $\rho_{t_i}'$ and $\rho_0$ converge to some unit vector $V\neq 0$ in $T_{\rho_0}Hit_n(S)'$ that is tangential to the Fuchsian locus.
\end{enumerate}

If (i) holds, then we have that $dL_\gamma(U)\geq 0$ for all $\gamma\in\Gamma$. On the other hand, if (ii) holds, then
\begin{eqnarray*}
dL_{\gamma}(U)  &=&  \frac{d}{dt} L_\gamma(\rho_t) |_{t=0} \\
&=&\lim_{i\to\infty}\frac{L_{\gamma}(\rho_{t_i}) - L_{\gamma}(\rho_{0})}{ t_i }\\
&\geq& \lim_{i \to \infty} \frac{L_{\gamma}(\rho_{t_i}') - L_{\gamma}(\rho_{0}')}{ t_i } \\
&=& \Big(\lim_{i\to \infty} s_i\Big) \cdot dL_{\gamma} (V)
\end{eqnarray*}
for some sequence of positive numbers $\{s_i\}_{i=1}^\infty$. Note that if $\lim_{i\to\infty}s_i=\infty$ then, $dL_{\gamma}(V)\leq 0$, which is impossible since $V\neq 0$ is tangential to the Fuchsian locus and all hyperbolic metrics on $S$ have the same area. Hence, $dL_{\gamma}(U+sV) \geq 0$ for some $s\leq0$. 

In either case, there is some vector $W\in T_{\rho_0} Hit_n(S)'$ (possibly the $0$ vector) that is tangential to the Fuchsian locus so that $dL_\gamma(U+W)\geq 0$ for all $\gamma$. Furthermore, since $U+W\neq 0$, the fact that the differentials $dL_\gamma$ generate the cotangent space of $Hit_n(S)'$ at $\rho_0$ implies that $dL_\gamma(U+W)> 0$ for some $\gamma$. By a similar argument, we can also show that there is some vector $W'\in T_{\rho_0} Hit_n(S)'$ that is tangential to the Fuchsian locus so that $dL_\gamma(-U+W')\geq 0$ for all $\gamma$, and this inequality holds strictly for some $\gamma$. 

Adding these two inequalities together gives $dL_\gamma(W+W')\geq 0$ for all $\gamma$, and $dL_\gamma(W+W')> 0$ for some $\gamma$. However, this is impossible since $W+W'$ is tangential to the Fuchsian locus.

\appendix
\section{Proof of Corollary \ref{proper}}\label{appendix}

In this appendix, we will prove the properness result stated as Corollary \ref{proper}. We begin by recalling some results from \cite{Zha1} that we will need. 

Let $\Pmc:=\{\gamma_1,\dots,\gamma_{3g-3}\}$ be an oriented pants decomposition of $S$, i.e. a maximal collection of pairwise non-intersecting, pairwise non-homotopic, homotopically non-trivial, oriented simple closed curves on $S$. These curves cut $S$ into $2g-2$ pairs of pants, which we label by $P_1,\dots,P_{2g-2}$, and also gives us a real analytic diffeomorphism
\[Hit_n(S)\to (\Rbbb^+)^{(3g-3)(n-1)}\times\Rbbb^{(3g-3)(n-1)}\times\Rbbb^{(2g-2)(n-1)(n-2)},\]
which one should think of as a generalization of the Fenchel-Nielsen coordinates on the Teichm\"uller space $\Tmc(S)$ (see Proposition 3.5 of \cite{Zha1}). 

The first $(3g-3)(n-1)$ positive numbers are called the \emph{boundary invariants}. For any $\rho\in Hit_n(S)$, these are the numbers
\[\beta_{\gamma_j,k}:=\log\bigg|\frac{\lambda_{k+1}(\rho(X_j))}{\lambda_k(\rho(X_j))}\bigg|,\]
where $k=1,\dots,n-1$ and $j=1,\dots,3g-3$. Here, $X_j\in\Gamma$ is a group element that corresponds to $\gamma_j$ and $\lambda_1(\rho(X_j)),\dots,\lambda_n(\rho(X_j))$ are the eigenvalues of $\rho(X_j)$ arranged in increasing order of moduli. Note that each of the $3g-3$ curves in $\Pmc$ are associated $n-1$ of these numbers. They capture the eigenvalue data of the holonomy about each of the curves in $\Pmc$, and are a generalization of the Fenchel-Nielsen length coordinates. 

The next $(3g-3)(n-1)$ real numbers are called the \emph{gluing parameters}, and these are also associated to the curves in $\Pmc$. Informally, the $n-1$ gluing parameters associated to each curve in $\Pmc$ is the data specifying how one should ``glue" the representations on adjacent pair of pants together along a common boundary component. Hence, these generalize the Fenchel-Nielsen twist coordinates. Just like the Fenchel-Nielsen twist coordinates, to specify these gluing parameters formally, we need to make additional topological choices to define what is ``zero gluing". In this case, this additional topological choice we make is a pair of distinct points $a_j,b_j\in\partial_\infty\Gamma$ so that $X_j^-$, $a_j$, $X_j^+$, $b_j$ lie in $\partial_\infty\Gamma$ in that cyclic order. 

For simplicity, we will fix such a choice once and for all in the following way. Let $P_1$, $P_2$ be the two pairs of pants that have $\gamma_j$ as a common boundary component (it is possible for $P_1=P_2$). For $i=1,2$, choose $A_i$, $B_i$ and $C_i$ be elements in $\Gamma$ corresponding the boundary components of $P_i$ so that $C_i\cdot B_i\cdot A_i=\id$ and $A_1=A_2^{-1}=X_j$. Let $a_j$ be the repelling fixed point of $B_1$ and $b_j$ be the repelling fixed point of $C_2$. The gluing parameters are then
\[g_{\gamma_j,k}:=\log\big(-(P_{k,1},P_{k,2},P_{k,4},P_{k,3})\big)\]
for $k=1,\dots,n-1$, where $\xi:\partial_\infty\Gamma\to\Fmc(\Rbbb^n)$ is the Frenet curve corresponding to $\rho$ and
\begin{eqnarray*}
P_{k,1}&:=&\xi(X_j^-)^{(k)}+\xi(X_j^+)^{(n-k-1)}\\
P_{k,2}&:=&\xi(X_j^-)^{(k-1)}+\xi(X_j^+)^{(n-k-1)}+\xi(a_j)^{(1)}\\
P_{k,3}&:=&\xi(X_j^-)^{(k-1)}+\xi(X_j^+)^{(n-k)}\\
P_{k,4}&:=&\xi(X_j^-)^{(k-1)}+\xi(X_j^+)^{(n-k-1)}+\xi(b_j)^{(1)}
\end{eqnarray*}
are four hyperplanes in $\Rbbb^n$ that intersect at $M_k:=\xi(X_j^-)^{(k-1)}+\xi(X_j^+)^{(n-k-1)}$.

Finally, the remaining $(2g-2)(n-1)(n-2)$ real numbers are called the \emph{internal parameters}, and are associated to the pairs of pants $P_1,\dots,P_{2g-2}$. To each $P_j$, we associate $(n-1)(n-2)$ internal parameters, and they parameterize the Hitchin representations on a pair of pants after the boundary invariants are fixed. These are defined in great detail in Section 3 of \cite{Zha1}. For our purposes though, we do not need to know what these parameters are, but only the following proposition.

\begin{prop}\label{internal}
Fix a pair of pants $P_{j_0}$ given by $\Pmc$. Let $\{\rho_i\}$ be a sequence in $Hit_n(S)$ so that 
\begin{itemize}
\item The boundary invariants corresponding to $\partial P_{j_0}$ remain bounded away from $0$ and $\infty$ along $\{\rho_i\}$.
\item Some internal parameter corresponding to $P_{j_0}$ grows to $\infty$ or $-\infty$ along $\{\rho_i\}$.
\end{itemize}
Let $\gamma$ be a closed curve in $S$ with the property that any closed curve homotopic to $\gamma$ has non-empty intersection with $P_{j_0}$. Then $\lim_{i\to\infty}l_{\rho_i}(\gamma)=\infty$.
\end{prop}

\begin{proof}
The proof of this proposition is a slight modification of the proof of the main theorem given in Section 5.1 of \cite{Zha1}. In Section 3.2 of \cite{Zha1}, there is a description of a particular way to cut each $P_j$ into two ideal triangles that share all three edges. Doing this over all $P_j$ gives us $6g-6$ edges. Here, we view each of these edges $e=[a,b]$ as a $\Gamma$-orbit of a pair of distinct points $a,b\in\partial_\infty\Gamma$. 

Let $\rho\in Hit_n(S)$ and $\xi$ the corresponding Frenet curve. As was done in Section 4.4 of \cite{Zha1}, one can associate a particular positive number $K[a,b]$ to each of these $6g-6$ edges $[a,b]$.  Using this, define
\[K(\rho,j_0):=\min\{K[a,b]:[a,b]\subset P_{j_0}\}.\]
The same argument as given in Section 5.1 of \cite{Zha1} proves that 
\[\lim_{i\to\infty}K(\rho_i,j_0)=\infty.\] 

Let $X\in\Gamma$ be a group element corresponding to $\gamma$ and let $X^+$ and $X^-$ be the attracting and repelling fixed points of $X$ in $\partial_\infty\Gamma$. Let $e=[a,b]$ be an edge in $P_{j_0}$ so that there is a lift $\etd=\{a,b\}$ with the property that $X^-,a,X^+,b$ lie in $\partial_\infty\Gamma$ in that cyclic order. Such an edge exists by the hypothesis we imposed on $\gamma$. For any $p=0,\dots,n-1$, one can define subsegments $c_p(\etd)$ of the projective line $\Pbbb(\xi(X^-)^{(1)}+\xi(X^+)^{(1)})\subset\Rbbb\Pbbb^{n-1}$ associated to each lift $\etd=\{a,b\}$ of $e=[a,b]$. These are called the \emph{crossing $(p)$-subsegments} (see Definition 4.7 of \cite{Zha1}). Using the cross ratio, we can define a notion of length for these subsegments, which we denote by $l(c_p(\etd))$ (see Definition 4.8 of \cite{Zha1}).

By the proof of Proposition 4.16 of \cite{Zha1}, we see that 
\[\frac{1}{n}\sum_{p=0}^{n-1}l(c_p(\etd))\geq K(\rho,j_0).\]
Furthermore, by Lemma 4.9 and Lemma 4.10 of \cite{Zha1}, we have
\[l_\rho(\gamma)\geq l(c_p(\etd)),\]
for all $p=0,\dots,n-1$, which allows us to conclude that 
\[l_\rho(\gamma)\geq K(\rho,j_0).\]
Combining this with the fact that $\lim_{i\to\infty}K(\rho_i,j_0)=\infty$  gives the proposition.
\end{proof}

With the above proposition, we are ready to prove Corollary \ref{proper}. Let $\{\rho_i\}$ be a sequence in $Hit_n(S)$, let $\Cmc:=\{\gamma_1,\dots,\gamma_k\}$ satisfy the hypothesis of Corollary \ref{proper} and let $\Pmc:=\{\gamma_1,\dots,\gamma_{3g-3}\}\subset\Cmc$ be a pants decomposition. Observe that the hypothesis on $\Cmc$ ensures the following:
\begin{itemize}
\item For any $\gamma\in\Pmc$, there is some $\gamma'\in\Cmc$ that intersects $\gamma$ transversely.
\item For each pair of pants $P$ given by $\Pmc$, there is some $\gamma\in\Cmc$ so that any closed curve homotopic to $\gamma$ has non-empty intersection with $P$.
\end{itemize}

The pants decomposition $\Pmc$ then gives us a parameterization of $Hit_n(S)$ as described above. We will prove Corollary \ref{proper} in the following steps. 
\begin{enumerate}
\item If there is some boundary invariant $\beta_{\gamma_j,k}$ so that $\lim_{i\to\infty}\beta_{\gamma_j,k}(\rho_i)=\infty$, then $\lim_{i\to\infty}l_{\rho_i}(\gamma_j)=\infty$.
\item If there is some boundary invariant $\beta_{\gamma_j,k}$ so that $\lim_{i\to\infty}\beta_{\gamma_j,k}(\rho_i)=0$, then $\lim_{i\to\infty}l_{\rho_i}(\gamma)=\infty$ for any closed curve $\gamma$ that intersects $\gamma_j$ transversely.
\item If all the boundary invariants remain bounded away from $0$ and $\infty$ and some internal parameter associated to a pair of pants $P$ grows to $\pm\infty$, then $\lim_{i\to\infty}l_{\rho_i}(\gamma)=\infty$ for any closed curve $\gamma$ with the property that any closed curve homotopic to $\gamma$ has non-empty intersection with $P$.
\item If all the boundary invariants remain bounded away from $0$ and $\infty$ and there is some gluing parameter $g_{\gamma_j,k}$ so that $\lim_{i\to\infty}g_{\gamma_j,k}(\rho_i)=\pm\infty$, then $\lim_{i\to\infty}l_{\rho_i}(\gamma)=\infty$ for any $\gamma$ that intersects $\gamma_j$ transversely.
\end{enumerate}
Note that together, the four statements above prove Corollary \ref{proper}. Statement (1) is obvious because 
\[l_\rho(\gamma_j)=\sum_{k=1}^{n-1}\beta_{\gamma_j,k}(\rho)\] 
and all the boundary invariants are positive. Also, Statement (3) is a restatement of Proposition \ref{internal}, and Statement (2) is an immediate consequence of (1) of Proposition \ref{main proposition}, which is a main result in this paper. The rest of this appendix will be the proof of Statement (4). 

Let $X_j,X\in\Gamma$ correspond to $\gamma_j$ and $\gamma$ respectively so that $X_j^-$, $X^-$, $X_j^+$, $X^+$ lie in $\partial_\infty\Gamma$ in that cyclic order. We previously chose a pair of points $a_j,b_j\in\partial_\infty\Gamma$ so that $X_j^-$, $a_j$, $X_j^+$, $b_j$ lie in $\partial_\infty\Gamma$ in that cyclic order in order to define the gluing parameters $g_{\gamma_j,k}$ associated to $\gamma_j$. If we choose $X_j^l\cdot a_j$ and $X_j^m\cdot b_j$ in place of $a_j$ and $b_j$, we get another collection of gluing parameters, which we denote by $g^{l,m}_{\gamma_j,k}$. The next lemma explains the relationship between $g_{\gamma_j,k}=g^{0,0}_{\gamma_j,k}$ and $g^{l,m}_{\gamma_j,k}$. Its proof is an easy computation which we omit.

\begin{lem}
Let $\rho\in Hit_n(S)$ and let $\lambda_1,\dots,\lambda_n$ be the eigenvalues of $\rho(X_j)$ arranged in increasing order of their moduli. For any integers $l,m$, we have 
\[g^{l,m}_{\gamma_j,k}(\rho)=(m-l)\log\bigg(\frac{\lambda_{k+1}}{\lambda_k}\bigg)+g_{\gamma_j,k}(\rho).\]
\end{lem}

In particular, when the boundary invariants corresponding to $\gamma_j$ are bounded away from $0$ and $\infty$ along a sequence of representations $\{\rho_i\}$ in $Hit_n(S)$, then $\lim_{i\to\infty}g_{\gamma_j,k}(\rho_i)=\pm\infty$ if and only if $\lim_{i\to\infty}g^{l,m}_{\gamma_j,k}(\rho_i)=\pm\infty$. Statement (4) then follows immediately from this observation and the following proposition.

\begin{prop}
Let $\rho\in Hit_n(S)$ and let $\gamma_j$, $\gamma$, $X_j^-$, $X_j^+$, $X^-$, $X^+$, $a_j$, $b_j$ be as above. Let $l$ and $m$ be integers such that $X_j^-$, $X_j^{l-1}\cdot a_j$, $X^-$, $X_j^l\cdot a_j$, $X_j^+$, $X_j^{m+1}\cdot b_j$, $X^+$, $X_j^m\cdot b_j$ lie in $\partial_\infty\Gamma$ in that cyclic order. Then
\[3l_\rho(\gamma)\geq g^{l,m}_{\gamma_j,k}(\rho)\,\,\,\,\,\text{ and }\,\,\,\,\, 3l_\rho(\gamma)\geq -g^{l-1,m+1}_{\gamma_j,k}(\rho)\]
for all $k=1,\dots,n-1$.
\end{prop}

\begin{proof}
The technique used in this proof is the same as that used in the proofs of Lemma 4.18 of \cite{Zha1}. For any $k=1,\dots,n-1$, let 
\begin{eqnarray*}
P_{k,0}&:=&\xi(X_j^-)^{(k)}+\xi(X_j^+)^{(n-k-1)},\\
P_{k,1}&:=&\xi(X_j^-)^{(k-1)}+\xi(X_j^+)^{(n-k-1)}+\xi(X^-)^{(1)},\\
P_{k,2}'&:=&\xi(X_j^-)^{(k-1)}+\xi(X_j^+)^{(n-k-1)}+\xi(X_j^{l-1}\cdot a_j)^{(1)},\\
P_{k,2}&:=&\xi(X_j^-)^{(k-1)}+\xi(X_j^+)^{(n-k-1)}+\xi(X_j^l\cdot a_j)^{(1)},\\
P_{k,3}&:=&\xi(X_j^-)^{(k-1)}+\xi(X_j^+)^{(n-k)},\\
P_{k,4}'&:=&\xi(X_j^-)^{(k-1)}+\xi(X_j^+)^{(n-k-1)}+\xi(X_j^{m}\cdot b_j)^{(1)},\\
P_{k,4}&:=&\xi(X_j^-)^{(k-1)}+\xi(X_j^+)^{(n-k-1)}+\xi(X_j^{m+1}\cdot b_j)^{(1)},\\
P_{k,5}&:=&\xi(X_j^-)^{(k-1)}+\xi(X_j^+)^{(n-k-1)}+\xi(X^+)^{(1)}.
\end{eqnarray*}
Also, for all $i=0,\dots,5$, let
\begin{eqnarray*}
L_{k,i}'&:=&P_{k,i}'\cap\big(\xi(X^-)^{(1)}+\xi(X^+)^{(1)}\big),\\
L_{k,i}&:=&P_{k,i}\cap\big(\xi(X^-)^{(1)}+\xi(X^+)^{(1)}\big)
\end{eqnarray*}
and let
\begin{eqnarray*}
L_{k,a_j}&:=&\big(\xi(X_j^{l-1}\cdot a_j)^{(k-1)}+\xi(X_j^l\cdot a_j)^{(n-k)}\big)\cap\big(\xi(X^-)^{(1)}+\xi(X^+)^{(1)}\big),\\
L_{k,b_j}&:=&\big(\xi(X_j^{m+1}\cdot b_j)^{(k-1)}+\xi(X_j^m\cdot b_j)^{(n-k)}\big)\cap\big(\xi(X^-)^{(1)}+\xi(X^+)^{(1)}\big).
\end{eqnarray*}

It follows from Lemma 2.5 of \cite{Zha1} that 
\[\xi(X^-)^{(1)}\,,\,\,\, L_{k,a_j}\,,\,\,\,L_{k,2}\,,\,\,\,L_{k,3}\,,\,\,\,L_{k,4}\,,\,\,\, L_{k,b_j}\,,\,\,\, \xi(X^+)^{(1)}\] 
lie in the projective line $\xi(X^-)^{(1)}+\xi(X^+)^{(1)}$ in that cyclic order. Also, by Lemma 4.11 of \cite{Zha1}, we know 
\[3l_\rho(\gamma)\geq\log\big(\xi(X^-)^{(1)},L_{k,a_j},L_{k,b_j},\xi(X^+)^{(1)}\big),\]
which implies that
\[3l_\rho(\gamma)\geq\log\big(\xi(X^-)^{(1)},L_{k,2},L_{k,3},\xi(X^+)^{(1)}\big)\]
and
\[3l_\rho(\gamma)\geq\log\big(\xi(X^-)^{(1)},L_{k,3},L_{k,4},\xi(X^+)^{(1)}\big)\]
by Lemma 2.9. Using Lemma 2.8 and Lemma 2.6, we can also deduce that
\begin{eqnarray*}
\big(\xi(X^-)^{(1)},L_{k,2},L_{k,3},\xi(X^+)^{(1)}\big)&=&\big(\xi(X^-)^{(1)},L_{k,2},L_{k,3},\xi(X^+)^{(1)}\big)_{M_k}\\ 
&=&(P_{k,1},P_{k,2},P_{k,3},P_{k,5}) \\
&\geq&(P_{k,0},P_{k,2},P_{k,3},P_{k,4}') \\
&=&1-(P_{k,0},P_{k,2},P_{k,4}',P_{k,3}) \\
&=&1+e^{g^{l,m}_{\gamma_j,k}}\\
&\geq&e^{g^{l,m}_{\gamma_j,k}}
\end{eqnarray*}
and 
\begin{eqnarray*}
\big(\xi(X^-)^{(1)},L_{k,3},L_{k,4},\xi(X^+)^{(1)}\big)&=&\big(\xi(X^-)^{(1)},L_{k,3},L_{k,4},\xi(X^+)^{(1)}\big)_{M_k}\\ 
&=&(P_{k,1},P_{k,3},P_{k,4},P_{k,5}) \\
&\geq&(P_{k,2}',P_{k,3},P_{k,4},P_{k,0}) \\
&=&1-\frac{1}{(P_{k,0},P_{k,2}',P_{k,4},P_{k,3})} \\
&=&1+e^{-g^{l-1,m+1}_{\gamma_j,k}}\\
&\geq&e^{-g^{l-1,m+1}_{\gamma_j,k}}
\end{eqnarray*}
where $M_k:=\xi(X_j^-)^{(k-1)}+\xi(X_j^+)^{(n-k-1)}$.
\end{proof}

\begin{multicols}{2}
\flushleft Gye-Seon Lee\\
Mathematisches Institut\\
Ruprecht-Karls-Universit\"{a}t Heidelberg\\
D-69120 Heidelberg, Germany\\
lee@mathi.uni-heidelberg.de\\

Tengren Zhang\\
Mathematics Department\\
University of Michigan\\
530 Church St, Ann Arbor, MI 48109\\
tengren@math.umich.edu
\end{multicols}

\end{document}